\documentclass[10pt,a4paper,twoside,reqno]{amsart}
\usepackage[french,british]{babel}  
\usepackage[T1]{fontenc} 
\usepackage{mathtools}
\usepackage{mathrsfs,amssymb}
\usepackage{stmaryrd}
\usepackage{bbm}
\usepackage[a4paper,top=3cm,bottom=3cm,inner=3cm,outer=3cm]{geometry}
\usepackage[all,2cell]{xy}\xyoption{v2}\UseAllTwocells\SilentMatrices
\usepackage{tikz}
\usetikzlibrary{positioning,decorations.pathreplacing,decorations.markings,arrows.meta,calc,fit,quotes,cd,math,arrows,backgrounds,shapes.geometric}
\usepackage{color}
\definecolor{darkgreen}{rgb}{0,0.45,0}
\usepackage[colorlinks,citecolor=darkgreen,urlcolor=gray,final,hyperindex,linktoc=page,hyperfootnotes=true]{hyperref}
\usepackage[capitalize]{cleveref}
\usepackage[inline]{enumitem}
\usepackage{comment}

\definecolor{mypurple}{rgb}{0.5, 0.0, 0.5}

\theoremstyle{plain}
\newtheorem{thm}{Theorem}[section]
\theoremstyle{plain}
\newtheorem{prop}[thm]{Proposition}
\theoremstyle{remark}
\newtheorem{rmk}[thm]{Remark}
\theoremstyle{plain}
\newtheorem{lem}[thm]{Lemma}
\theoremstyle{plain}
\newtheorem{cor}[thm]{Corollary}
\theoremstyle{definition}
\newtheorem{defi}[thm]{Definition}
\theoremstyle{definition}
\newtheorem{ex}[thm]{Example}

\newcommand{\ca}{\mathcal}
\newcommand{\caa}{\mathcal}

\newcommand{\Hom}{\ensuremath{\mathrm{Hom}}}
\newcommand{\HOM}{\textrm{\scshape{Hom}}}
\newcommand{\Comod}{\ensuremath{\mathsf{Comod}}}
\newcommand{\Mod}{\ensuremath{\mathsf{Mod}}}

\newcommand{\Alg}{\ensuremath{\mathsf{Alg}}}
\newcommand{\Coalg}{\ensuremath{\mathsf{Coalg}}}
\newcommand{\Mon}{\ensuremath{\mathsf{Mon}}}
\newcommand{\Comon}{\ensuremath{\mathsf{Comon}}}

\newcommand{\Cart}{\ensuremath{\mathrm{Cart}}}
\newcommand{\Cocart}{\ensuremath{\mathrm{Cocart}}}

\newcommand{\B}{\mathbf}

\DeclareMathOperator*\Meas{Meas}
\newcommand{\HS}{\operatorname{HS}}
\newcommand{\Der}{\operatorname{Der}}
\newcommand{\MDer}{\operatorname{MDer}}
\newcommand{\op}{\mathrm{op}}

\newcommand{\id}{\mathrm{id}}

\newcommand{\Cat}{\ensuremath{\mathsf{Cat}}}
\newcommand{\ot}{\otimes}
\newcommand{\ringk}{\Bbbk}
\tikzset{tick/.style={postaction={decorate,decoration={markings,mark=at position 0.5 with {\draw[-] (0,.4ex) -- (0,-.4ex);}}}}}
\newcommand{\tickar}{\begin{tikzcd}[baseline=-0.5ex,cramped,sep=small,ampersand replacement=\&]{}\ar[r,tick]\&{}\end{tikzcd}}

\title{Measuring comodules and enrichment}
\author[M. Hyland]{Martin Hyland}
\address{Department of Pure Mathematics and Mathematical Statistics,
  University of Cambridge, UK.}
\email{M.Hyland@dpmms.cam.ac.uk}
\author[I. L\'opez Franco]{Ignacio L\'opez Franco}
\address{Department of Mathematics and Applications, CURE, Universidad de la
  Rep\'ublica, Maldonado, Uruguay.}
\email{ilopez@cure.edu.uy}
\thanks{The second author was partially supported by SNI--ANII and PEDECIBA}
\author[C. Vasilakopoulou]{Christina Vasilakopoulou}
\address{School of Applied Mathematical and Physical Sciences, National Technical University of Athens, Greece.}
\email{cvasilak@math.ntua.gr}
\thanks{The third author was supported in the framework of H.F.R.I call ``3rd Call for H.F.R.I.'s
Research Projects to Support Faculty Members \& Researchers'' (H.F.R.I. Project Number: 23249).}

\begin{document}
\selectlanguage{british}  
\begin{abstract}
  This paper extends the theory of universal measuring comonoids
  to modules and comodules in braided monoidal categories. We generalise the
  universal measuring comodule $Q(M,N)$, originally introduced for modules over
  $\ringk$-algebras when $\ringk$ is a field, to arbitrary braided monoidal categories. In
  order to establish its existence, we prove a representability theorem for
  presheaves on opfibred categories and an adjoint functor theorem for opfibred
  functors. The global categories of modules and comodules, fibred and opfibred
  over monoids and comonoids respectively, are shown to exhibit an enrichment of
  modules in comodules. Additionally, we use our framework to study higher
  derivations of algebras and modules, defining along the way the
  non-commutative Hasse-Schmidt algebra.
\end{abstract}

\maketitle

\tableofcontents

\section{Introduction}\label{Introduction}

The Sweedler dual, or finite dual, of an associative and unital $\ringk{}$-algebra $A$ is
the $\ringk{}$-coalgebra $A^\circ$ characterized by the property that coalgebra
morphisms $C \to A^\circ$ are in natural bijection with algebra morphisms
$A \to C^*$. More generally, the Sweedler hom of two $\ringk{}$-algebras $A$ and $B$ is
a $\ringk{}$-coalgebra $P(A,B)$ with the universal property that coalgebra morphisms
$C \to P(A,B)$ correspond bijectively to algebra morphisms $A \to \Hom_\ringk{}(C,B)$
into the convolution algebra of $\ringk{}$-linear maps from $C$ to $B$. This
construction was shown to exist when $\ringk{}$ is a field in~\cite{Sweedler}.

Subsequent work extended these ideas beyond fields.
When $\ringk{}$ is a commutative ring, the Sweedler dual of a $\ringk{}$-algebra was studied
in~\cite{Sweedlerdual}. A broader generalization to (co)monoids in a braided
monoidal category was developed in~\cite{Measuringcomonoid}, where a detailed account of the enrichment of the category of monoids in the
category of comonoids was provided, using the universal measuring
comonoids $P(A,B)$ as the hom-objects.

This article builds on this prior work by shifting focus to modules and
comodules.
When $\ringk{}$ is a field, the \emph{universal measuring comodule} $Q(M,N)$, for modules $M$ and
$N$ over $\ringk{}$-algebras, was introduced in~\cite{Batchelor}.
It satifies the universal property that module morphisms $M \to \Hom_\ringk{}(X,N)$
are in one-to-one correspondence with comodule morphisms $X \to Q(M,N)$. These
objects have found applications in areas such as connections on bundles, loop
algebras, and representation theory \cite{Batchelor}.

Our main objective is to investigate measuring comodules within a braided monoidal
category, establish existence results for the universal measuring comodule in
this setting, and demonstrate an enrichment of the category of modules in the
category of comodules. To this end, we use the \emph{global category of modules},
whose objects are pairs $(A,M)$ consisting of a monoid $A$ and an $A$-module,
and the \emph{global category of comodules}, defined dually. The former is
a category fibred over the category of monoids, while the latter is opfibred
over the category of comonoids. This structure prompts us to explore the
interplay between the universal measuring comonoid and the (op)fibration
structures. In particular, we prove a representability theorem for presheaves on an opfibred
category and an adjoint functor theorem for opfibred functors.

As a key result,
we show that the global category of modules is enriched in the global category of
comodules. Moreover, in the symmetric monoidal setting, the enrichment is that of a
monoidal category.
This is accomplished by using the relationship between actions of monoidal
categories and enriched categories.

We end the paper with an application of measuring comonoids and comodules to the
study of higher derivations, or Hasse-Schmidt derivations~\cite{zbMATH03027033},
of algebras and modules. For instance, we exhibit the Hasse-Schmidt algebra of
a ring extension as a kind of colimit---a tensor product--- in the coalgebra-enriched
category of algebras, both in the non-commutative and commutative context.

The article is organized as follows: Section~\ref{background} reviews key
concepts, including (co)monoids and (co)modules in monoidal categories, local
presentability, (opmonoidal) actions inducing (monoidal) enrichment, and the
construction of the universal measuring comonoid. In
Section~\ref{sec:fibrations}, we examine conditions for the existence of
adjoints to fibred 1-cells between (op)fibrations over arbitrary
bases. Section~\ref{sec:enrichmentModComod} describes the global categories of
modules and comodules and defines the universal measuring comodule in that
context. After a short example about derivations,
Section~\ref{Universalmeasuringcomodule} establishes the existence of the
universal measuring comodule, and show that its existence implies that of the
universal measuring comonoid. Section~\ref{sec:enrichmentofmodincomod} derives the
enrichment of modules in comodules, while Section~\ref{sec:coinv-univ-meas}
identifies the object of coinvariants of the comonoid $Q(M,N)$ with an object of
maps from $M$ to $N$.
Finally,
Section~\ref{sec:higher-derivations} explores higher derivations of algebras and
modules through the lens of measuring comodules.

\section{Background}\label{background}

In this section, we recall some of the main concepts and constructions
needed for the development of the current work. In particular, we will
summarize some of the key results from \cite{Measuringcomonoid}
pertinent to this paper. We assume familiarity with
the basics of the theory of monoidal categories, found for example in
\cite{BraidedTensorCats}, and enriched categories~\cite{Kelly}.

\subsection{(Co)monoids and (co)modules}\label{ComonoidsComodules}

Suppose $(\ca{V},\otimes,I)$ is a monoidal category. A \emph{monoid} is an object $A$ equipped with
a multiplication $m\colon A\otimes A\to A$ and unit $\eta\colon I\to A$ that
satisfy the usual associativity
and unit laws; along with monoid morphisms, they form a category $\Mon(\ca{V})$.
Dually, we have \emph{comonoids} $(C,\delta\colon C\to C\otimes C,\epsilon\colon C\to I)$ whose category
is denoted by $\Comon(\ca{V})$. Both these categories are monoidal
if $\ca{V}$ is braided monoidal, and when the braiding is a symmetry they inherit the symmetry.
Sometimes we will write simply $\Mon$ or $\Comon$ when the monoidal category
$\ca{V}$ is understood.

If $F\colon\ca{V}\to\ca{W}$ is a lax monoidal functor, with structure maps $\phi_{A,B}\colon FA\otimes FB\to F(A\otimes B)$
and $\phi_0\colon I\to F(I)$, it induces a map between their categories
of monoids $\Mon F\colon\Mon(\ca{V})\to\Mon(\ca{W})$
by $(A,m,\eta)\mapsto(FA,Fm\circ\phi_{A,A},F\eta\circ\phi_0)$. Dually,
oplax monoidal functors induce maps between the categories of comonoids.

For functors between monoidal categories, standard doctrinal adjunction
arguments \cite{Doctrinal} imply that
oplax monoidal structures on a left adjoint correspond bijectively to lax
monoidal structures on the corresponding right adjoint.
This generalizes to parametrized adjunctions, as found in
\cite[Prop.~3.2.3]{PhDChristina} or for higher dimension in
\cite[Prop.~2]{Monoidalbicats&hopfalgebroids}. Therefore, if
$\ca{V}$ is braided monoidal closed, the internal hom
functor $[-,-]\colon\ca{V}^\op\times\ca{V}\to\ca{V}$ acquires a lax monoidal
structure as the parametrized adjoint of the strong monoidal tensor product functor $(-\otimes-)$.
The induced functor between the monoids is denoted by
\begin{equation}\label{defMon[]}
[-,-]\colon\Comon(\ca{V})^\op\times\Mon(\ca{V})\to\Mon(\ca{V});
\end{equation}
for $C$ a comonoid and $A$ a monoid, $[C,A]$ has the \emph{convolution} monoid structure.

A \emph{(right) $A$-module} for a monoid $A$ in $\ca{V}$
is an object $M$ of $\ca{V}$ equipped with an
arrow $\mu\colon M\otimes A\to M$ called
the \emph{action}, such that the diagrams
\begin{equation*}\label{defmod}
\begin{tikzcd}
M \otimes A \otimes A \arrow[r, "1 \otimes m"] \arrow[d, "\mu \otimes 1"'] 
& M \otimes A \arrow[d, "\mu"] \\
M \otimes A \arrow[r, "\mu"'] 
& M
\end{tikzcd}
\qquad\mathrm{and}\qquad
\begin{tikzcd}
& M \otimes A \arrow[rd, "\mu"] & \\
M\otimes I \arrow[rr, "r_M"',"\cong"] \arrow[ur, "1\otimes \eta"] & & M
\end{tikzcd}
\end{equation*}
commute. An \emph{$A$-module morphism} $(M,\mu)\to(M',\mu')$
is an arrow $f\colon M\to M'$ in $\ca{V}$ such that
$\mu'\circ(f\otimes A)=f\circ\mu$. For any monoid $A$ in $\ca{V}$,
there is a category $\Mod(\ca{V})_A$
of right $A$-modules
and $A$-module morphisms. Dually, we have a category of right $C$-comodules
$\Comod(\ca{V})_C$ for every $C\in\Comon(\ca{V})$.
Although we will use right (co)modules throughout the paper, there are analogous
presentations using various combinations of left or right (co)modules.
We will omit the base
monoidal category $\ca{V}$ from the notation when there is no room for ambiguity.

As well as inducing a functor $\Mon F$ between the categories of monoids, a lax monoidal functor $F\colon \ca{V}\to\ca{W}$
induces functors $\Mod F\colon \Mod(\ca{V})_A\to \Mod(\ca{W})_{FA}$
where the $FA$-action on $FM$ is
$F\mu\circ\phi_{M,A}\colon FA\otimes FM\to
F(A\otimes M)\to FM$.
In particular, the lax monoidal functor $[-,-]\colon \ca{V}^\mathrm{op}\times\ca{V}\to\ca{V}$
in a braided monoidal closed category induces a functor
\begin{equation} \label{defMod[]}
[-,-]\colon  \Comod_C^\mathrm{op}\times\Mod_A\to\Mod_{[C,A]}
\end{equation}
for any comonoid $C$ and monoid $A$.

Each monoid morphism $f\colon A\to B$ determines a \emph{restriction of scalars} functor 
$f^*\colon \Mod_B\to\Mod_A$
which makes every $B$-module $(N,\mu)$ into an $A$-module
$f^* N$ via the action
$\mu(f\otimes N)\colon N\otimes A\to
N\otimes B\to N$.
This functor commutes with the forgetful functors into $\ca{V}$.
Dually, each comonoid morphism $g\colon C\to D$ gives rise to a
\emph{corestriction of scalars} functor
$g_!\colon \Comod_C\longrightarrow\Comod_D$
which commutes with the forgetful functors into $\ca{V}$.
Notice that $f^*$ preserves all limits and $g_!$
all colimits that exist in $\ca{V}$.

\subsection{Local presentability}

In this section we collect some facts about locally presentable
categories. 
Recall that a category is $\kappa$-\emph{filtered}, for a regular cardinal $\kappa$,
if each subcategory with less than $\kappa$ arrows is the base of a co-cone. A $\kappa$-filtered colimit is a colimit of a functor whose domain is a $\kappa$-filtered category.
For a regular cardinal $\kappa$, a \emph{$\kappa$-accessible} category $\ca{C}$ is a
category, with a small set 
of $\kappa$-presentable objects
(\emph{i.e.}, objects $C$ such that $\ca{C}(C,-)$ preserves $\kappa$-filtered colimits) such that every object
in $\ca{C}$ is the $\kappa$-filtered colimit of $\kappa$-presentable objects.
An \emph{accessible category} is one that is $\kappa$-accessible for some $\kappa$.
A \emph{locally presentable} category is an
accessible category that is cocomplete.
A functor between accessible categories is \emph{accessible} if it preserves
$\kappa$-filtered colimits, for some regular cardinal $\kappa$. 
Adjoint functors between accessible categories are accessible.
We refer the reader to \cite{MakkaiPare,LocallyPresentable}
for more on the theory of locally presentable categories.

A \emph{locally presentable monoidal category} is a monoidal category $\ca{V}$
whose underlying category is locally presentable and whose tensor
product is an accessible functor. (The latter condition is guaranteed when, for
example, the monoidal category is biclosed.) In this context,
the categories $\Mon(\ca{V})$ and $\Comon(\ca{V})$ are both locally presentable.
This result can be found in \cite[\S~2]{MonComonBimon}, and in fact it follows from
the much more general `Limit Theorem' \cite[5.1.6]{MakkaiPare} since
both categories can be constructed as 2-categorical limits of accessible
functors; see also~\cite[Prop.~2.9]{Measuringcomonoid}.
Notice that (co)monoids inherit (co)completeness from $\ca{V}$ without any extra assumptions, see for example \cite[\S~2.6]{MonComonBimon}.

Regarding the categories of (co)modules over a (co)monoid, those are also locally presentable due the
following well-known result. We include a proof outline for the sake of completeness.

\begin{thm}\label{Monadiccomonadicpresentability}
  The category of (co)algebras for an accessible (co)monad on a locally
  presentable category is locally presentable.
\end{thm}
\begin{proof}
  The accessiblity of
  the category of (co)algebras again follows from the Limit Theorem~\cite[5.1.6]{MakkaiPare}, since
  these categories are limits \cite[p.~101]{MakkaiPare}; see also~\cite[2.78]{LocallyPresentable}.
  The category of coalgebras for a comonad on the category $\ca{C}$ is
  cocomplete if $\ca{C}$ is cocomplete. Locally presentability is ensured in
  this case. The category of algebras for an accessible monad on $\ca{C}$ is complete
  because $\ca{C}$ is. It is, therefore, locally presentable by~\cite[Thm.~6.1.4]{MakkaiPare}.
\end{proof}

In the case of the (co)monad given by tensoring on the right with a
(co)monoid, we have the following (compare with \cite{CoringsComod}).
\begin{cor}\label{comodlocpresent}
Suppose $\ca{V}$ is a locally presentable
monoidal category (so $\otimes$ is accessible). Then
$\Mod_A$ and $\Comod_C$
are locally presentable categories, for any monoid $A$ and any comonoid $C$ in $\ca{V}$.
\end{cor}

An important fact which will be used repeatedly is that any cocontinuous functor
with domain a locally presentable category has a right adjoint; this can be obtained as
a corollary to the following adjoint functor theorem, since the set of
$\kappa$-presentable objects of a locally $\kappa$-presentable category
form a small dense subcategory.

\begin{thm}\cite[5.33]{Kelly}\label{Kelly}
If the cocomplete $\ca{C}$ has a small dense subcategory, every
cocontinuous $S\colon \ca{C}\to\ca{B}$ has a right adjoint.
\end{thm}

As an application, we can deduce the following, proved in \cite[\S~2.II]{Measuringcomonoid}.

\begin{prop}\label{Comonclosed}
For $\ca{V}$ a locally presentable braided monoidal closed category, $\Comon(\ca{V})$ is
comonadic over $\ca{V}$, and also monoidal closed. We denote its internal hom by
\begin{equation}\label{def:HOM}
{\HOM}\colon\Comon(\ca{V})^\op\times\Comon(\ca{V})\to\Comon(\ca{V}).
\end{equation}
\end{prop}

\subsection{Actions and enrichment}\label{sec:actionenrich}


In this section we briefly recall the relationship between actions of a monoidal
category and enriched categories. This was established for bicategories
in~\cite{enrthrvar} and can also be found in the case of monoidal categories
in~\cite{AnoteonActions}.

A left action of a monoidal category $\ca{V}$ on a category $\ca{D}$ is a
functor $*\colon \ca{V}\times\ca{D}\to\ca{D}$ together with natural isomorphisms
$X*(Y*D)\cong(X\otimes Y)*D$ and $D\cong I*D$, for $X$, $Y\in\ca{V}$ and
$D\in\ca{D}$, that satisfy a pentagon axiom and a unit axiom, similar to those
of a monoidal category. Categories equipped with an action of a monoidal
category are also referred to as \emph{actegories} in the literature
(see e.g. \cite{Mccruddencoalgebroidsreps}).

A monoidal category $\mathcal{V}$ acts on itself through its tensor
product. When $\mathcal{V}$ is monoidal left closed, the opposite category
$\mathcal{V}^{\mathrm{op}}$ acts on $\mathcal{V}$ via the internal hom. By left
closed, we mean that for each object $X$, the functor $-\otimes X$ admits a
right adjoint $[X,-]$, noting that terminology may differ across the literature.

The following two theorems give conditions
under which an action induces an enrichment, and
an opmonoidal action induces a monoidal enrichment.
To be precise, the composition of a $\ca{V}$-category $\ca{A}$ will consist of
morphisms $\ca{A}(B,C)\otimes\ca{A}(A,B)\to\ca{A}(A,C)$, the same convention as in
Kelly's book \cite{Kelly}.

\begin{thm}\label{actionenrich}
  Suppose that $\ca{V}$ is a monoidal category acting on a category $\ca{D}$ via
  a functor $*\colon \ca{V}\times\ca{D}\to\ca{D}$, such that the functor $-*D$
  has a right adjoint $F(D,-)$ for every $D\in\ca{D}$.
  Then $\ca{D}$ can be enriched
  in $\ca{V}$, meaning there exists a $\ca{V}$-category $\underline{\ca{D}}$
  with hom-objects $\underline{\ca{D}}(A,B)=F(A,B)$ and underlying category
  $\ca{D}$.

  Moreover, if $\ca{V}$ is left closed, then $\underline{\ca{D}}$ is tensored,
  with $X*D$ serving as the tensor product of $X\in\ca{V}$ and $D\in\ca{D}$. If
  $\ca{V}$ is also symmetric, then $\underline{\ca{D}}$ is cotensored provided
  the functors $X*-$ have right adjoints. Finally, the opposite category
  $\ca{D}^\op$ can also be enriched in $\ca{V}$.
\end{thm}

The above follows from a much stronger result of \cite{enrthrvar}
regarding categories enriched in bicategories; details can be found
in \cite{AnoteonActions} and \cite[\S~4.3]{PhDChristina}.
Although the symmetry in the statement is not essential, it has the advantage of
settling us within the framework of Kelly's book \cite{Kelly}, a straightforward
scenario in which to consider cotensor products.

In some instances, the enriched category $\underline{\ca{D}}$ from the previous
theorem has a monoidal structure, as provided by the following result.

\begin{thm}[{\cite[Thm.~3.6 and 3.10]{Measuringcomonoid}}]\label{opmonactionmonenrich}
  Assume that the monoidal category $\ca{V}$ in Theorem~\ref{actionenrich} is
  braided, $\ca{D}$ is a monoidal category and the action is opmonoidal. Then, the induced enriched
  category $\underline{\ca{D}}$ is monoidally $\ca{V}$-enriched, with underlying
  monoidal category $\ca{D}$.
  Furthermore, $\underline{\ca{D}}$ is braided when $\ca{D}$ and the action are braided.
\end{thm}
Let us expound the statement of the theorem. An opmonoidal action of
the braided monoidal category $\ca{V}$ on a monoidal category $\ca{D}$ is
defined as an action $*\colon \ca{V}\times \ca{D}\to\ca{D}$ equipped with an
opmonoidal structure, where the structure isomorphisms are opmonoidal
transformations. For details, see \cite[Definition~3.5]{Measuringcomonoid}. When
we say that the monoidal category $\ca{D}$ is monoidally
$\ca{V}$-enriched we mean that its tensor product underlies a $\ca{V}$-functor
$\underline{\ca{D}}\otimes\underline{\ca{D}}\to\underline{\ca{D}}$ and its
structure natural transformations are $\ca{V}$-natural. When the monoidal
category $\ca{D}$ is braided and $\ca{V}\times\ca{D}\to\ca{D}$ is braided
opmonoidal, then \cite[3.10]{Measuringcomonoid} shows that the braiding is
$\ca{V}$-natural.

When the monoidal category $\ca{V}$ is symmetric monoidal closed, then the
action $[-,-]\colon\ca{V}^\op\times\ca{V}\to\ca{V}$ is a braided lax monoidal functor,
and the structural isomorphisms $[X,[Y,Z]]\cong[X\otimes Y,Z]$ and $Z\cong
[I,Z]$ are monoidal transformations. This, along with the following lemma, was
shown in \cite{Measuringcomonoid}.

\begin{lem}\label{inthomaction}
Suppose $\ca{V}$ is a braided monoidal closed category. The internal hom
$[-,-]\colon \ca{V}^\mathrm{op}\times\ca{V}\to\ca{V}$ is a monoidal action of
$\ca{V}^\mathrm{op}$ on $\ca{V}$, and it induces an action
of $\Comon(\ca{V})^\mathrm{op}$ on $\Mon(\ca{V})$. When $\ca{V}$ is symmetric,
then this action is braided lax monoidal.
\end{lem}

The symmetry hypothesis in the lemma is not only used to give a braided lax
monoidal structure on $[-,-]$ but also to guarantee that the braiding lifts to
$\Comon(\ca{V})$ and $\Mon(\ca{V})$.

We will frequently employ the dual form of the lemma, in which
$[-,-]^{\mathrm{op}}$ defines an opmonoidal action of $\mathcal{V}$ on
$\mathcal{V}^{\mathrm{op}}$, and this action extends to an opmonoidal action of
$\Comon(\mathcal{V})$ on $\Mon(\mathcal{V})^{\mathrm{op}}$.

\subsection{Universal measuring comonoids}\label{unimeascom}

A fundamental objective of \cite{Measuringcomonoid} was to establish an
enrichment of the category of monoids over the category of comonoids.
The key results are outlined below, with
further details available in Sections~4 and 5 of \cite{Measuringcomonoid}. Certain aspects of this theory were generalised in the many-object setting of enriched
categories in \cite{VCocats}.

Suppose that $(C,\delta,\varepsilon)$ is a comonoid and $(A,m,\eta)$, $(B,m,\eta)$ are monoids in a braided monoidal
category. A \emph{measuring} is a morphism $\phi\colon A\otimes C\to B$ that makes the
following two diagrams commutative (where $c$ is the braiding).
\begin{equation}\label{measuring_mon}
  \begin{tikzcd}[column sep=.6in]
  A\ot A\ot C \arrow[r, "A\ot A\ot\delta"] \arrow[d, "m\ot C"'] &
  A\ot A\ot C\ot C \arrow[r, "A\ot c_{A,C}\ot C"] &
  A\ot C\ot A\ot C \arrow[d, "\phi\ot \phi"] \\
  A\ot C \arrow[r, "\phi"] &
  B &
  B\ot B \arrow[l, "m"']
\end{tikzcd}
\quad
\begin{tikzcd}
  C\ar[r,"\varepsilon"]\ar[d,"\eta\otimes C"']&I\ar[d,"\eta"]\\
  A\otimes C\ar[r,"\phi"]& B
\end{tikzcd}
\end{equation}

One says that $C$ is a \emph{measuring comonoid.}
The sets $\Meas(A,C,B)$ of measurings as above are the values of a functor
$\Meas{}:\Mon(\ca{V})^{\mathrm{op}}\times\Comon(\ca{V})^{\mathrm{op}}\times
\Mon(\ca{V})\to\mathsf{Set}$. A \emph{universal measuring comonoid} for two
monoids $A$ and $B$, denoted by $P(A,B)$, is a representation of the presheaf
$\Meas(A,-,B)$ on $\Comon(\ca{V})$; so
$\Comon(\ca{V})(-,P(A,B))\cong\Meas(A,-,B)$.
We note that universal measuring $\ringk{}$-coalgebras, for a commutative
ring $\ringk{}$, were considered in~\cite{BarrCoalgebras}, for
algebras but also for algebras for a PROP, which includes the case of
associative algebras.

When the braided monoidal category $\ca{V}$ is
closed, the two diagrams say that the morphism $\hat \phi\colon A\to [C,B]$ associated to $\phi$ is
a morphism of monoids, so $\Meas(A,C,B)\cong\Mon(\ca{V})(A,[C,B])$ naturally in
all three variables. Therefore, we have:

\begin{thm}\cite[Thm.~4.1]{Measuringcomonoid}\label{measuringcomonoidprop}
If $\ca{V}$ is locally presentable braided monoidal closed category, the functor
$[-,B]^\op\colon\Comon(\ca{V})\to\Mon(\ca{V})^\op$ has a right adjoint
$P(-,B)$, i.e.
there is a natural isomorphism
\begin{equation}
\Mon(\ca{V})(A,[C,B])\cong\Comon(\ca{V})(C,P(A,B)).\label{eq:12}
\end{equation}
\end{thm}
The parametrized adjoint $P\colon\Mon(\ca{V})^\op\times\Mon(\ca{V})\to\Comon(\ca{V})$
of $\Mon[-,-]$ is called the \emph{Sweedler hom} functor.
When
$\ca{V}$ is the category of vector spaces over a field, and $A$ is a
$\ringk{}$-algebra, $P(A,\ringk{})$ is the well-known  \emph{Sweedler} or \emph{finite dual} $A^\circ$ of $A$; see \cite{Sweedler}.

Moreover, the functor $[C,-]^\op\colon\Mon(\ca{V})^\op\to\Mon(\ca{V})^\op$ for a comonoid $C$ has a right adjoint $(C\triangleright-)^\op$, and the functor
of two variables
$\triangleright\colon\Comon(\ca{V})\times\Mon(\ca{V})\to\Mon(\ca{V})$ is called the \emph{Sweedler product}
in \cite{AnelJoyal}.

By applying Theorems~\ref{actionenrich} and \ref{opmonactionmonenrich}
to the action of $\Comon(\ca{V})$ on $\Mon(\ca{V})^\op$ given by the
internal hom functor $[-,-]$ considered in Lemma~\ref{inthomaction},
we obtain a category $\mathcal{M}$ enriched in $\Comon(\ca{V})$ with underlying
category $\Mon(\ca{V})^\op$ and hom-objects $\mathcal M(A,B)=P(B,A)$. Taking the
opposite enriched category, we have:

\begin{thm}\label{monoidenrichment}
Suppose $\ca{V}$ is a locally presentable symmetric monoidal closed category.
 The category $\Mon(\ca{V})$ is a monoidal $\Comon(\ca{V})$-category,
 tensored and cotensored, with $\underline{\Mon(\ca{V})}(A,B)=P(A,B)$,
 cotensor $[C,B]$ and tensor $C\triangleright B$ for any comonoid $C$ and monoid $B$.
\end{thm}

Before concluding this section, we note a straightforward observation that will prove helpful later. 
For a pair of monoid measurings $\phi\colon A\otimes C\to B$ and $\phi'\colon A'\otimes C'\to B'$ as previously defined, it may be desirable for the morphism
\begin{equation}
  \label{eq:46}
  \phi\bullet\phi'\colon
  A\otimes A'\otimes C\otimes C'
  \cong
  A\otimes C\otimes A'\otimes C'
  \xrightarrow{\phi\otimes\phi'}
  B\otimes B'
\end{equation}
to also be a measuring, where the isomorphism arises from the braiding.
This, however, only holds if the braiding is a symmetry.
\begin{lem}
  \label{l:7}
  In a symmetric monoidal category,
  the morphism $\phi\bullet\phi'$ is a measuring.
\end{lem}
The proof of the lemma consists of writing the axioms of a measuring for
$\phi\bullet\phi'$ and noticing that the symmetry of the braiding is required.
In the case of a symmetric monoidal \emph{closed} category, the lemma has an intuitive
intepretation. If $\hat\phi$ and $\hat\phi'$ are the monoid morphisms corresponding
to $\phi$ and $\phi'$, respectively, then (\ref{eq:46}) correponds to
\begin{equation*}
  A\otimes A'\xrightarrow{\hat\phi\otimes\hat\phi'}[C,B]\otimes[C',B']
  \xrightarrow{\chi}[C\otimes C',B\otimes B']
\end{equation*}
where $\chi$ is part of the lax monoidal structure on the internal hom induced
by the braiding.  We can guarantee that $\chi$ is a morphism of monoids only when
this lax monoidal structure is braided, which is to say, that the monoidal
structure on the tensor product is braided, or that the braiding is a
symmetry.

We close the section with one of the central examples of measuring from
\cite{MeasuringCoalgebras} making the connection with derivations, of which we
recall the definition below.

If $M$ is a bimodule over a $\ringk{}$-algebra $A$
over a commutative ring $\ringk{}$, a \emph{derivation}, or \emph{ordinary
  derivation} to distinguish them from higher derivations, is a $\ringk{}$-linear
morphism $\delta\colon A\to M$ that satisfies the
Leibniz rule
$\delta(ab)=a\cdot\delta(b)+\delta(a)\cdot b$, for $a$, $b\in A$. When $f\colon A\to B$ is
a $\ringk{}$-algebra morphism, then $B$ can be regarded as an $A$-bimodule via $f$ and
derivations $\delta\colon A\to B$ will be called $f$-derivations. The pairs
$(f,\delta)$ form a set $\Der(A,B)$.

If $C$ is a coalgebra over a field $\ringk{}$, with comultiplication $\Delta$
and counit $\varepsilon$, an element $g\in C$ is \emph{group-like} if the
correponding linear map $\ringk{}\to C$ is a morphism of coalgebras; that is,
$\Delta(g)=g\otimes g$ and $\varepsilon(g)=1$. An element $x\in C$ is \emph{$g$-primitive} if
$\Delta(x)=g\otimes x+x\otimes g$ (it follows that $\varepsilon(x)=0$).

\begin{ex}
  \label{ex:3}
  Denote by $C_1$ the dual coalgebra of the two-dimensional $\ringk{}$-algebra of dual
  numbers $\ringk{}[\epsilon]\cong \ringk{}[x]/(x^2)$. If $\{g,x\}$ is the dual basis of
  $\{1,\epsilon\}\subset \ringk{}[\epsilon]$, we have that $g$ is a group-like element
  and $x$ is a $g$-primitive element.
  This coalgebra is the representing object of the functor
  $\Coalg\to\mathsf{Set}$ that sends each coalgebra
  $C$ to the set of all the pairs $(g,x)$ where $g$ is a group-like element and
  $x$ is $g$-primitive. 
  Batchelor~\cite{Differenceoperators} showed that there is bijection between
  measuring maps $A\otimes C_1 \to B$ and elements of the
  set $\operatorname{Der}(A,B)$, where $B$ is regarded as an
  $A$-bimodule via restriction of scalars along $f$; see above.
  From the defining property of the universal measuring coalgebra, we obtain a
  canonical isomorphism
  \begin{equation*}
    \Coalg(C_1,P(A,B))\cong \operatorname{Der}(A,B).
  \end{equation*}
  The adjunction isomorphism $\Coalg(C_1,P(A,B))\cong\Alg(A,[C_1,B])$ yields the
  well-known natural isomorphism $\operatorname{Der}(A,B)\cong \Alg(A,B[x]/(x^2))$.

\end{ex}

\section{Adjoints to fibred 1-cells}\label{sec:fibrations}

This section assumes familiarity with the basic notions of (op)fibrations, or
(op)fibred categories, originally introduced in
\cite{Grothendieckcategoriesfibrees}, a good account of which can be found
in~\cite{streicher2023fiberedcategorieslajean}.

If $U\colon \ca{C}\to\caa{X}$ and $V\colon \ca{D}\to\caa{X}$ are opfibrations
and $K\colon \ca{C}\to\ca{D}$ satisfies $VK=U$ and preserves cocartesian
morphisms, then the fibrewise
adjunctions $K_X\dashv R_X\colon\ca{D}_X\to\ca{C}_X$, for $X\in\caa{X}$, paste
into an adjunction $K\dashv R$ where $UR=V$, see e.g. first part of \cite[Prop.~8.4.2]{Handbook2}.
 Here $R$ need not preserve
cocartesian morphisms. Indeed, one defines $R$ on a cocartesian morphism $h\colon
B\to f_!(B)$ over $f\colon X\to X'$ as $R_X(B)\to f_!R_X(B)\to R_{X'}f_!(B)$
where the first morphism is a cocartesian lifting of $f$ and the second is the
component at $B$ of the natural transformation that is
the mate of $K_{X'}f_!\cong f_!K_X$ via $K_X\dashv R_X$ and $K_{X'}\dashv
R_{X'}$. 

One is naturally led to consider the question of when
for an opfibred 1-cell
\begin{equation}
  \label{eq:startingshape}
  \begin{tikzcd}
    \ca{C}\ar[r,"K"]\ar[d,"U"']&
    \ca{D}\ar[d,"V"]\\
    \caa{X}\ar[r,"F"]&
    \caa{Y}
  \end{tikzcd}
\end{equation}
a right adjoint for $F$ and fibrewise right adjoints for $K$ induce an
adjunction between the total categories. Related questions on existence of adjoints between fibrations can be found in \cite{FibredAdjunctions}.

Starting with a representability result (Lemma~\ref{l:5}), in this section we give what seems to be
the shortest proof (Corollary~\ref{cor:3}) that avoids two-dimensional category theory.
Finally, we investigate when a right adjoint for $K$ gives rise to a right
adjoint for $F$.
Before all that, we begin by recalling the basic definitions involved.

\subsection{Basic definitions}\label{fibrationsbasicdefinitions}

All our (op)fibrations will be equipped with a choice of (co)cartesian liftings,
usually known as a cleavage. We choose to de-emphasise the alternative
description of cloven (op)fibrations as indexed categories.

If $P\colon \ca{A}\to\caa{X}$ is a functor, we denote by $\ca{A}_X$ the fibre of $P$
over $X\in\caa{X}$: the subcategory of
$\ca{A}$ defined by all the morphisms that are mapped by $P$ to $1_X$.
Recall that $P$ is a (cloven) \emph{fibration} if for all $f\colon X\to Y$ in
$\caa{X}$ and $B\in\ca{A}_Y$, there is a (chosen) cartesian lifting denoted by $\tilde{f}=\Cart(f,B)\colon f^*(B)\to B$, and dually for an opfibration. 
The category $\ca{A}$ is the \emph{total} category,
$\caa{X}$ is the \emph{base} category.
Any arrow in the total category of a fibration factorises uniquely into
a vertical arrow followed by a cartesian one, and dually for opfibrations.

 For every morphism $f\colon X\to Y$ in the base $\caa{X}$ of a cloven
 fibration, there is \emph{reindexing functor}
$f^*\colon \ca{A}_Y\to\ca{A}_X$
mapping each object to the domain of the chosen cartesian lifting along $f$.
It can be verified that $1_{\ca{A}_X}\cong(1_A)^*$ and that for composable morphisms in the base category,
$(g\circ f)^*\cong g^*\circ f^*$. If these isomorphisms are equalities, we have the notion of a \emph{split} fibration.

A \emph{fibred 1-cell} $(S,F)\colon P\to Q$ between two fibrations $P\colon \ca{A}\to\caa{X}$
and $Q\colon \ca{B}\to\caa{Y}$
is a commutative square
\begin{equation}\label{commutativefibredcell}
\begin{tikzcd}[row sep=small]
\ca{A} \arrow[r, "S"] \arrow[d, "P"'] & \ca{B} \arrow[d, "Q"] \\
\caa{X} \arrow[r, "F"] & \caa{Y}
\end{tikzcd}
\end{equation}
where $S$ preserves cartesian arrows.
In particular, when $P$ and $Q$ are fibrations over
the same base, a \emph{fibred functor} is a fibred 1-cell $(S,1_{\caa{X}})$.
Dually, we have the notions of an \emph{opfibred 1-cell}
and \emph{opfibred functor}.
Any commutative diagram~\eqref{commutativefibredcell} gives rise to a collection of functors
$S_X\colon\ca{A}_X\to \ca{B}_{FX}$
between the fibre
categories.

Given two fibred 1-cells $(S,F),(T,G)\colon P\rightrightarrows Q$,
a \emph{fibred 2-cell} from $(S,F)$ to
$(T,G)$ is a pair of natural transformations 
$(\alpha\colon S\Rightarrow T,\beta\colon F\Rightarrow G)$
with $\alpha$ lying above $\beta$, \emph{i.e.}, $Q\alpha
=\beta P$.
Dually, we have the notion of an \emph{opfibred 2-cell}.

There is a 2-category $\B{Fib}$ of
fibrations over arbitrary base categories,
fibred 1-cells and fibred 2-cells, and dually, a 2-category
$\B{OpFib}$ of opfibrations.
Both are equipped with a forgetful functor to the 2-category
$\Cat^{\mathbbm{2}}$ of 2-functors $\mathbbm{2}\to\Cat$, 2-natural
transformations and modifications.

\subsection{A representability result}\label{fibredadjunctions}

Recall that a representation of a presheaf
$\phi\colon\ca{C}^{\mathrm{op}}\to\mathsf{Set}$ is a pair $(C,x)$ where $C$ is an object
of $\ca{C}$ and $x\in\phi(C)$, and the corresponding natural transformation
$\ca{C}(-,C)\Rightarrow\phi$ is an isomorphism. 

The following lemma gives necessary and sufficient conditions
for a presheaf $\phi$ on a total category of an opfibration
$U\colon\ca{C}\to\caa{X}$ to be representable
by an object over $X\in\caa{X}$
in terms of a certain `cartesian' natural transformation
$\beta\colon\phi\Rightarrow\caa{X}(U-,X)$ satisfying a pullback property.
We will denote by $J\colon\ca{C}_{X}\hookrightarrow\ca{C}$ the inclusion
functor, by $\Delta\caa{X}(X,X)$ the presheaf on $\ca{C}_X$ that is constant at
the set $\ca{X}(X,X)$, and by $1_X\colon 1 \Rightarrow\Delta\caa{X}(X,X)$ the natural transformation
from the terminal presheaf to the presheaf on $\ca{C}_X$ constant at the set
$\caa{X}(X,X)$ with all components equal to the identity $1_X$.

\begin{lem}
  \label{l:5}
  With the notation above,
  a presheaf $\phi\colon\ca{C}^\op\to\mathsf{Set}$ is representable by $(C,x)$ with $C\in\ca{C}_X$ if and
  only if there exists a natural transformation
  $\beta\colon\phi\Rightarrow\caa{X}(U-,X)$ satisfying:
  \begin{enumerate}
  \item \label{item:5} 
  There is a pullback square in $[\ca{C}^\op_X,\mathsf{Set}]$
    \begin{equation}
      \label{eq:23}
      \begin{tikzcd}
        \ca{C}_{X}(-,C)\ar[r,"{\pi}"]
        \ar[d]
        &
        \phi J^{\mathrm{op}}
        \ar[d,"{\beta J^{\mathrm{op}}}"]
        \\
        1
        \ar[r,"1_{X}"]
        &
        \Delta\caa{X}(X,X)\ar[ul,phantom,"\lrcorner", very near end]
      \end{tikzcd}
    \end{equation}
    where $\pi_C(1_C)=x$;
  \item \label{item:6} $\beta$ is a cartesian natural transformation when
    restricted to cocartesian morphisms; \emph{i.e.}, the naturality squares of
    $\beta$ corresponding to cocartesian morphisms in $\ca{C}$ are pullbacks.
  \end{enumerate}
\end{lem}
\begin{proof}
  First suppose that $\phi$ is of the form $\ca{C}(-,C)$ with $U(1_C)=X$. Define
  $\beta$ as the natural transformation corresponding to $1_X\in
  \caa{X}(U(1_C),X)$. This is, $\beta_{C'}(f)=Uf$.
  For any $C'\in\ca{C}_X$, the
  square~\eqref{eq:23} now looks like
  \begin{equation}
    \label{eq:37}
    \begin{tikzcd}
      \ca{C}_X(C',C)\ar[r,hook]\ar[d]&\ca{C}(C',C)\ar[d,"U"]\\
      1\ar[r,"1_X"]&\caa{X}(X,X)
    \end{tikzcd}
  \end{equation}
  where the right vertical arrow sends $f$ to $Uf$. This square is clearly a
  pullback. Requiring that $\beta$ be a cartesian natural transformation on
  cocartesian morphisms is simply restating the universal property of
  cocartesian morphism (at least, those with codomain $C$).

  In the case that $\phi$ is represented by $(C,x)$, then $\beta$ is defined as
  \begin{equation*}
    \beta\colon \phi\cong \ca{C}(-,C)\xrightarrow{U}\caa{X}(U-,X), \qquad
    \beta_C(x)=1_X.
  \end{equation*}
  which
  is cartesian with respect to cocartesian morphisms since natural isomorphisms
  are cartesian transformations. The pullback square~\eqref{eq:23} is obtained
  from the square~\eqref{eq:37} by composing with the isomorphism
  $\ca{C}(-,C)\cong \phi$.

  We shall now prove the converse: that the two conditions from the statement
  imply that $(C,x)$ represents $\phi$, where
  $x=\pi_{C}(1_{C})$.
  Given
  $x'\in\phi(C')$, we have to find a unique $g\colon C'\to C$ with
  $\phi(g)(x)=x'$.

  Pick a cocartesian lifting $h\colon C'\to \tilde X$ of $\beta_{C'}(x')\colon
  U(C')\to X$. By hypothesis we have a pullback square as
  depicted, and thus a unique element $\tilde x\in\phi(\tilde X)$ mapped to
  elements as depicted on the right.
  \begin{equation}
    \label{eq:27}
    \begin{tikzcd}[column sep=huge]
      \phi(\tilde X)
      \ar[r,"\phi(h)"]
      \ar[d,"\beta_{\tilde X}"]
      &
      \phi(C')
      \ar[d,"\beta_{C'}"]
      \\
      \caa{X}(X,X)\ar[r,"{\caa{X}(\beta_{C'}(x'),1)}"]
      &
      \caa{X}(U(C'),X)
    \end{tikzcd}
    \qquad
    \begin{tikzcd}
      \tilde x
      \ar[r,mapsto]\ar[d,mapsto]
      &
      x'
      \ar[d,mapsto]
      \\
      1_X
      \ar[r,mapsto]
      &
      \beta_{C'}(x)
    \end{tikzcd}
  \end{equation}

  The first hypothesis yields a pullback as shown below, and therefore, a unique
  $v\colon \tilde X\to C$ in $\ca{C}_{X}$ such that $\pi_{\tilde X}(v)=\tilde x$.
  \begin{equation*}
    \label{eq:28}
    \begin{tikzcd}
      \ca{C}_X(\tilde X,C)\ar[r,"\pi_{\tilde X}"]
      \ar[d]
      &
      \phi(\tilde X)\ar[d,"\beta_{\tilde X}"]
      \\
      1\ar[r,"1_{X}"]
      &
      \caa{X}(X,X)
    \end{tikzcd}
  \end{equation*}
  By naturality of $\pi$ we have
  \begin{equation*}
    \label{eq:29}
    \phi(v)(x)=\phi(v)(\pi_C(1_C))=\pi_{\tilde X}(1_C\circ v)= \pi_{\tilde
      X}(v)=\tilde x,
  \end{equation*}
  so $v\circ h\colon C'\to \tilde X\to C$ satisfies
  $\phi(v\circ h)(x)=\phi(h)\phi(v)(x)=\phi(h)(\tilde x)=x'$.

  This deals with existence so it remains to show uniqueness. Given
  $g\colon C'\to C$ in $\ca{C}$ such that $\phi(g)(x)=x'$, we have to show
  $g=v\circ h$. To start with we notice that
  \begin{equation*}
    \label{eq:30}
    \beta_{C'}(x')=\beta_{C'}(\phi(g)(x))=\beta_{C}(x)\circ U(g) =U(g)
  \end{equation*}
  where we use the naturality of $\beta$ and the fact that $\beta_C(x)=1_X$ by
  definition of $x$. Therefore, $g=w\circ h$, where $h\colon C\to \tilde X$ is
  the cocartesian lifting of $\beta_{C'}(x')$ as above, and $w$ is a unique
  morphism in $\ca{C}_X$. It remains to show that $w=v$. Now $\pi_{\tilde X}$ is a
  monomorphism, as by Definition~\eqref{eq:23} it is the pullback of a monomorphism; so it suffices
  to show that
  \begin{equation}
    \pi_{\tilde X}(w)=\pi_{\tilde X}(v).\label{eq:33}
  \end{equation}
  Both sides of~\eqref{eq:33} have the same image under $\beta_{\tilde X}$,
  namely, $1_X$, by definition of $\pi$. Therefore, since the projections of a
  pullback are jointly monomophic, it suffices to prove that both
  sides of~\eqref{eq:33} have the same image under $\phi(h)$; the pullback
  involved here is that in~\eqref{eq:27}.

  The naturality of $\pi$ gives
  $\pi_{\tilde X}(w)=\phi(w)(\pi_C(1_C))=\phi(w)(x)$, so
  \begin{equation*}
    \label{eq:31}
    \phi(h)(\pi_{\tilde X}(w))=
    \phi(h)\phi(w)(x)=\phi(w\circ h)(x)=\phi(g)(x)=x'
  \end{equation*}
  and for the same reason $\phi(h)(\pi_{\tilde X}(v))=x'$. Therefore, $w=v$ and
  the proof is complete.
\end{proof}

Let us look at an illustrative example of Lemma~\ref{l:5}.
\begin{ex}
  \label{ex:2}
Consider the presheaf
$\operatorname{Cone}(-,D)$ that assigns to each object $C$ of $\mathcal{C}$ the
set of cones with vertex $C$ and with base the functor of small domain
$D\colon\mathcal{D}\to\ca{C}$.
Suppose
that $UD\colon\mathcal{D}\to\caa{X}$ has a limiting cone $\sigma_d\colon X\to
UD(d)$.
Define the natural transformation $\beta$ from the previous lemma with components
$\beta_C$ sending a cone $\xi_d\colon C\to D(d)$ to the unique $U(C)\to X$ that
composed with $\sigma_d$ equals $U\xi_d$.
Suppose that 
there is a cone $\tau_d\colon L\to D(d)$ with $P\tau_d=\sigma_d$
which is universal among those cones, as in condition \ref{item:5} in the lemma; that is, each cone $\rho_d\colon
C\to D(d)$ with $C\in\ca{C}_X$ is of the form $\rho_d=\tau_d h$ for a vertical
morphism $h$, which is unique among vertical morphisms.
The hypothesis \ref{item:6} of the lemma is automatically satisfied in this
example, something that can be verified using the definition of
cocartesian morphism.
Then, the converse part of the lemma's statement asserts that $\tau$ is a limiting cone for $D$.
\end{ex}

\subsection{An adjoint functor theorem}\label{sec:1}
\begin{thm}
  \label{thm:3}
  Suppose given an opfibred 1-cell $(K,F)$ as in~(\ref{eq:startingshape}) and
  $X\in\caa{X}$ a coreflection of $Y\in\caa{Y}$ along $F$ with counit
  $\varepsilon\colon FX\to Y$. The the following are equivalent for objects
  $C\in\ca{C}_X$ and
  $D\in\ca{D}_Y$.
  \begin{enumerate}
  \item \label{item:a}
    $C$ is a coreflection of $D$ along $K$ with
    counit $e\colon K(C)\to D$ such that $Ve=\varepsilon$.
  \item \label{item:b}
    $C$ is a coreflection of $D$ along the functor
    \begin{equation*}
      \ca{C}_{X}\xrightarrow{K}\ca{D}_{F(X)}\xrightarrow{\varepsilon_!}\ca{D}_Y.
    \end{equation*}
  \end{enumerate}
  In this case, $e=u\circ \tilde\varepsilon_C$, where $\tilde\varepsilon_C$ is a
  cocartesian lifting of $\varepsilon_C$, and $e\colon KC\to D$ and $u\colon
  \varepsilon_!KC\to D$ are the counits of the respective coreflections.
\end{thm}
\begin{proof}
  The proof is an application of Lemma~\ref{l:5}.
  Assuming that $\ca{D}(K-,D)$ is representable by $C$ with universal element
  $e\colon K(C)\to D$, the natural transformation $\beta$ constructed in the
  lemma is
  \begin{equation*}
    \beta\colon
    \ca{D}(K-,D)\xRightarrow{V} \caa{Y}(FU-,Y)\cong \caa{X}(U-,X),
    \qquad \beta_C(e)=1_X.
  \end{equation*}
  Then the lemma yields that the presheaf on $\ca{C}_X$ given by $C\mapsto
  \{f\in \ca{D}(K(C),D): Vf=\varepsilon\}$ is representable with universal
  element by $e$. This is the same as saying that the isomorphic presheaf
  $C\mapsto \ca{D}_Y(\varepsilon_!K(C),D)$ is representable with univesal
  element the vertical component in the cocartesian-vertical factorization of
  $e$.

  Now assume \ref{item:b} and set $\phi=\ca{D}(K-,D)$ and
  \begin{equation*}
    \beta\colon \ca{D}(K-,D)\xRightarrow{V}\caa{X}(FU-,Y)\cong\caa{X}(U-,X).
  \end{equation*}
  If $h\colon C''\to C'$ is a cocartesian morphism in $\ca{C}$, then $K(h)$ is a
  cocartesian morphism in $\ca{D}$, which is the same as asserting that the
  following is a pullback square. Then, the second hypothesis in Lemma~\ref{l:5}
  is satisfied:
  \begin{equation*}
    \begin{tikzcd}[column sep = large]
      \ca{D}(K(C'),D)
      \ar[d,"V"]\ar[r,"{\ca{D}(Kh,1)}"]
      &
      \ca{D}(K(C''),D)
      \ar[d,"V"]
      \\
      \caa{Y}(VK(C''),V(D))
      \ar[r,"{\caa{Y}(VKh,1)}"]
      &
      \caa{Y}(VK(C'),V(D))
    \end{tikzcd}
  \end{equation*}
  A choice of a cocartesian lifting
  $\tilde\varepsilon_{C'}\colon K(C')\to \varepsilon_!K(C')$ for each $C'\in
  \ca{C}_{X}$ gives rise to a natural transformation $\tilde\varepsilon\colon
  KJ\Rightarrow \varepsilon_!KJ\colon\ca{C}_X\to\ca{D}$ which makes the square
  on the right hand side below a pullback.
  \begin{equation*}
    \begin{tikzcd}
      \ca{D}_Y(\varepsilon_!KJ-,D)
      \ar[r]\ar[d]
      &
      \ca{D}(\varepsilon_!KJ-,D)
      \ar[r,"{\ca{D}(\tilde\varepsilon,1)}"]
      \ar[d,"V"]
      &
      \ca{D}(KJ,D)
      \ar[d,"V"]
      \\
      {1}
      \ar[r,"{1_{Y}}"]
      &
      \caa{Y}(Y,Y)
      \ar[r,"{\caa{Y}(\varepsilon,1)}"]
      &
      \caa{Y}(F(X),Y)
    \end{tikzcd}
  \end{equation*}
  The square on the left hand side is a pullback, by definition of the fibre
  $\ca{D}_Y$. The presheaf on $\ca{C}_X$ on the top left of the diagram
  is representable by hypothesis, so the outer diagram is the pullback square
  required in Lemma~\ref{l:5} (\ref{item:5}). So by the lemma, $\ca{D}(K-,D)$ is
  representable and the universal morphism
  $K(C)\to D$ of the representation is given by composition of the universal morphism
  $\varepsilon_!K(C)\to D$ with $\tilde\varepsilon_{C}\colon K(C)\to
  \varepsilon_!K(C)$.
\end{proof}

\begin{cor}\label{cor:4}
  Suppose given an opfibred 1-cell $(K,F)$ as in~(\ref{eq:startingshape}) and
  that $F$ has a right adjoint $G$ with counit $\varepsilon\colon FG\Rightarrow
  1_{\caa{Y}}$. The the following are equivalent.
  \begin{enumerate}
  \item \label{item:c} $K$ has a right adjoint $R$ such that:
    \begin{enumerate*}
    \item $UR(D)=GV(D)$ for all $D\in\ca{D}$;
    \item the counit $e\colon KR\Rightarrow 1_{\mathcal{D}}$ satisfies
      $Ve=\varepsilon U$.
    \end{enumerate*}
  \item \label{item:d}
    The functor
    \begin{equation}
      \label{eq:39}
      \ca{C}_{G(Y)}\xrightarrow{K}\ca{D}_{FG(Y)}\xrightarrow{\varepsilon_!}\ca{D}_Y
    \end{equation}
    has a right adjoint $R_Y$, for all $Y\in\caa{Y}$.
  \end{enumerate}
  In this case:
  \begin{enumerate}[label=(\alph*)]
  \item $R_Y$ is, up to unique isomorphism, the restriction of $R$ to fibres.
  \item \label{item:e}
    $e_D=u_D\circ \tilde \varepsilon_D$ where
    $\tilde\varepsilon_D\colon KR_Y(D)\to\varepsilon_!KR_Y(D)$ is the cocartesian
    lifting, and $e_D$ and $u_D$ are the components of the
    counit of $K\dashv R$ and $\varepsilon_!K\dashv R_Y$, respectively. 
  \item \label{item:f}
    $(K,F)\dashv(R,G)$ in $\Cat^{\mathbbm{2}}$.
  \end{enumerate}
\end{cor}
\begin{proof}
  The proof follows from Theorem~\ref{thm:3} and the fact that a functor
  has a right adjoint precisely when each object of its codomain category has a
  coreflection.   

  Assume~\ref{item:c}, so we have a coreflection  along
  $K$, namely $R(D)$,  in the fibre of $GV(D)$, for each $D\in\mathcal D$,
  providing the right adjoint $R$. The theorem provides us with a reflection of
  $D$ along~\eqref{eq:39}. Since this is for all $D\in \ca{D}_{V(D)}$ we showed
  \ref{item:d} in the case $Y=V(D)$. The special case when $D_Y$ is the empty
  category is trivial, since \eqref{eq:39} must be the identity functor.

  We next check that $R$ restricts to fibres. Given $t\colon D\to D'$ in
  $\ca{D}_Y$, we have $e_D'\circ KR(t)=t \circ e_D$, and applying $V$ we obtain
  $\varepsilon_{V(D')}\circ FUR(t)=V(t)\circ\varepsilon_{V(D)}$. Therefore
  $UR(t)=GV(t)$. But $V(t)=1$, so $R(t)$ is a vertical morphism. By construction
  $R_Y$ is the restriction of $R$ to fibres, so we have~\ref{item:e}. The
  equality $Ve=\varepsilon U$ says that $(e,\varepsilon)$ is a 2-cell in
  $\Cat^{\mathbbm{2}}$. This is sufficient to have an adjunction in this
  2-category as in~\ref{item:f} (the condition on the unit automatically holds).

  Now assume~\ref{item:d}. The previous theorem tells us that $R_Y(D)$ is a
  correflection of $D$ along $K$ with counit $e_D=u_D\circ \tilde\varepsilon_D$
  as in \ref{item:e}. This gives a right adjoint $R$ to $K$ with counit of
  components $e_D$, and we proved the statement~\ref{item:c}.
\end{proof}
Although not needed for what follows, for completeness purposes we state the
following theorem that adds to Corollary~\ref{cor:4} a necessary and sufficient
condition for the right adjoint to be an opfibred 1-cell. For a full proof, see
\cite[\S~5.3]{PhDChristina}.

\begin{thm}\label{totaladjointthm}
Suppose $(K,F)\colon U\to V$ is an opfibred 1-cell and $F\dashv G$ is an adjunction between the bases as in (\ref{eq:startingshape}).
Then, $(K,F)$ has an opfibred right adjoint $(R,G)$ if and only if, for
each $Y\in\caa{Y}$, there is an adjunction $(\varepsilon_Y)_!K_{GY}\dashv R_Y$ and the mate
\begin{equation*}
\begin{tikzcd}
\mathcal{D}_Y \arrow[r, "R_Y"] \arrow[d, "h_!"'] & \mathcal{C}_{GY} \arrow[d,
"(Gh)_!"]
\arrow[dl, Rightarrow, shorten >=1ex,shorten <=1ex,"\omega"']\\
\mathcal{D}_W \arrow[r, "R_W"']  & \mathcal{C}_{GW}
\end{tikzcd}
\end{equation*}
of $(\varepsilon_W)_!K_{GW}(Gh)_! \cong
(\varepsilon_W)_!(FGh)_!K_{GY} \cong
h_!(\varepsilon_Y)_!K_{GY}$
is an isomorphism.
\end{thm}

\subsection{Existence of a right adjoint between the base categories}

In the preceding section, we examined the conditions that allow the total category
component $K$ of an opfibred 1-cell $(K, F)$ to have a right adjoint, provided the
base space component F does. Here, we address the reverse inquiry: does $F$
possess a right adjoint if $K$ does?

In the next couple of lemmas we will use R.~Guitart's exact
squares~\cite{zbMATH03779585}, which we now briefly describe. Recall that any functor $F\colon\ca{A}\to\ca{B}$ gives rise to two profunctors, $F_*\colon\ca{A}\tickar\ca{B}$ given by
$\ca{B}(-,F-)$ and $F^*\colon\ca{B}\tickar\ca{A}$ given by $\ca{B}(F-,-)$. Moreover, it is the case that $F_*\dashv F^*$ in the bicategory of profunctors.

Suppose that $\ca{A},\ca{B},\ca{C},\ca{D}$ are categories, $T,W,Z,S$ functors and $\phi$ a natural transformation as depicted below.
Such a square filled with a natural transformation is called \emph{exact} if the morphism of profunctors
$W_*T^*\Rightarrow S^*Z_*$ that is the mate of $\varphi_*$ is invertible.
\begin{equation}
  \label{eq:21}
  \begin{tikzcd}
    \mathcal{A} \arrow[r, "T"] \arrow[d, "W"'] & \mathcal{B} \arrow[d, "Z"] \\
    \mathcal{C} \arrow[r, "S"] & \mathcal{D}
    \arrow[Rightarrow, from=2-1, to=1-2, shorten >= 1ex,shorten <=1ex,"\varphi"]
  \end{tikzcd}
\end{equation}
This can be translated in more elementary terms by stating that the morphisms
\begin{equation*}
  \int^{A\in\mathcal{A}}\mathcal{C}(C,W(A))\times\mathcal{B}(T(A),B)
  \to \mathcal{D}(S(C),Z(B)),
\end{equation*}
given by sending $f\colon C\to W(A)$ and $g\colon T(A)\to B$ to $Z(g)\varphi_A
S(f)$, are invertible.

In the case when there are adjunctions $W^\ell\dashv W$ and $Z^\ell\dashv Z$,
the mate $\varphi^\#\colon Z^\ell S\Rightarrow TW^\ell$ is given by
$$
\varphi^\#\colon
Z^\ell S\Rightarrow Z^\ell SWW^\ell \xrightarrow{Z^\ell \varphi T^\ell}
Z^\ell ZTW^\ell\Rightarrow TW^\ell
$$
where the unlabelled arrows are induced by the unit of $W^\ell\dashv W$ and the
counit of $Z^\ell \dashv Z$.
In this situation, (\ref{eq:21}) is exact if and only if $\varphi^\#$ is
invertible. For, $W_*\cong (W^\ell)^*$ and $Z_*\cong (Z^\ell)^*$, so the mate
of $\varphi_*$ is $(\varphi^\#)^*\colon (TW^\ell)^*\Rightarrow(Z^\ell S)^*$,
which is invertible if and only if $\varphi^\#$ is so.

A particular case of interest of an exact square is when $\ca{A}=\mathbf{1}$ in
\eqref{eq:21}. In this situation $W$ and $Z$ can be regarded as objects and
$\varphi$ has only one component
$\varphi\colon S(W)\to Z$. The exactness means that
$\ca{C}(-,W)\cong\ca{D}(S(-),Z)$ as presheaves, which is to say that $W$ is a
coreflection of $Z$ along $S$ with counit $\varphi$.

\begin{lem}
  \label{l:extsq-adj}
  Assume that the square~\eqref{eq:21} is exact.
  Then, if $A$ is a coreflection of $B$ along $T$, with
  counit $e\colon T(A)\to B$, then $W(A)$ is a coreflection of $Z(B)$ along $S$
  with counit $Z(e)\circ \varphi_B\colon SW(A)\to ZT(A)\to Z(B)$. In particular,
  $S$ has a right adjoint provided that $T$ does and that $Z$ is essentially
  surjective on objects.
\end{lem}
\begin{proof}
  As pointed out in the paragraph preceding this lemma, the coreflection
  $A$ of $B$ corresponds to an exact square, with identity top side and 2-cell
  $\varphi$. This square can be pasted with the square \eqref{eq:21}, giving
  rise to a pasted exact square (exact squares are closed under pasting
  \cite{zbMATH03779585}). Therefore we have another exact square with top side
  the identity functor of $\mathbf{1}$, which is another way of looking at the
  coreflection of the statement.

  The final claim in the statement becomes evident when $Z$ is surjective on
  objects: the presence of a right adjoint corresponds to the existence of a
  coreflection for every object in $\ca{D}$ , each of which is of the form
  $V(B)$ for some object $B$ in $\ca{B}$ . If $Z$ is merely essentially
  surjective, this scenario can be reduced to the previous case using standard
  methods.
\end{proof}

In what follows, the dual of \cite[Prop.~4.4]{Grayfibredandcofibred} will be
useful. It states that an opfibration has a left adjoint with identity unit if
and only if we can choose an initial object in each fibre, and the functors of
change of fibre preserve initial objects.
\begin{lem}
  \label{l:squareisexact}
  Suppose that the domain $U$ and the codomain $V$ in the opfibred 1-cell
  \eqref{eq:startingshape} have initial objects preserved under fibre change.
  If $K$ preserves initial
  objects in fibres,
  then  the square \eqref{eq:startingshape} is exact.
\end{lem}
\begin{proof}
  Denote by $U^\ell\dashv U$ and $V^\ell\dashv V$ the adjunctions given by the
  comments above this lemma, both with identity unit.
  We have to show that the mate of the
  identity transformation, namely
  $V^\ell F =V^\ell FU U^\ell=V^\ell VKU^\ell\Rightarrow KU^\ell$, is
  an isomorphism.
  In other words we have to show that the components $\beta_{KU^\ell(C)}\colon V^\ell
  VKU^{\ell}(C)\to KU^{\ell}(C)$ of
  the counit $\beta$ of $V^\ell \dashv V$ are invertible. These components are
  vertical morphisms due to the triangle identity $V\beta=1$.
  Recall that $U^\ell(C)$ is an initial object in the fibre over $C$.
  On the one hand, the
  domain of this morphism is an initial object of the fibre of $V$ over
  $VKU^\ell(C)=FUU^\ell(C)=F(C)$.  On the other hand, its codomain is an initial
  object in the fibre of $V$ over $F(C)$, since $K$ preserves these initial
  objects. Therefore, $\beta_{KU^\ell(C)}$ is invertible.
\end{proof}

\begin{thm}
  \label{thm:1}
  Suppose that, in an opfibred 1-cell $(K,F)\colon U\to V$
  as in \eqref{eq:startingshape}, $U$ and $V$ have initial objects preserved under
  fibre change and $K$ preserves initial objects on
  fibres.
  \begin{enumerate}
  \item \label{item:11}
    If $C\in \mathcal{C}$ is a coreflection along $K$ of an object $D\in\ca{D}$,
    with counit $e\colon K(C)\to D$,
    then $U(C)$ is a coreflection along $F$ of $V(D)$ with counit $V(e)$.
  \item \label{item:12}
    $F$ has a right adjoint $G$ whenever $K$ does so. Moreover, $GV\cong UR$ and
    $G\cong URV^\ell$ where $V^\ell\dashv V$ and $K\dashv R$.
  \end{enumerate}
\end{thm}
\begin{proof}
  The first part and the existence of a right adjoint $G$ of $F$ follow directly
  from Lemmas~\ref{l:extsq-adj} and \ref{l:squareisexact}, as $V$ is surjective
  on objects. So $G(Y)$ is defined by $URV^\ell(Y)$, for $Y\in\caa{Y}$.

  Let us first verify that $UR$ sends vertical morphisms to isomorphisms.
  If $v\colon D\to D'$ is a morphism in $\ca{D}$, then
  $e_{D'}\circ KR(v)=v\circ e_{D}$ where $e\colon KR\Rightarrow 1$ is the counit
  of $K\dashv R$. If $v$ is vertical, applying $V$ we obtain $V(e_{D'})\circ
  FUR(v)=V(e_D)$.
  But both $V(e_{D'})$ and $V(e_{D})$ exhibit, respectively, $UR(D)$ and
  $UR(D')$ as a coreflection of $V(D)=V(D')$ along $F$, by
  Lemma~\ref{l:extsq-adj}. Then $UR(v)$ is an isomorphism.

  Denote by $\beta\colon V^\ell V\Rightarrow 1_{\caa{Y}}$ the counit of the
  adjunction $V^\ell\dashv V$. The components $\beta_D\colon 0_{V(D)}\to D$ are
  the unique vertical morphisms from the initial objects of the fibres
  $\mathcal{D}_{V(D)}$. Then $GV=URV^\ell V\cong UR$ via $UR\beta$.
\end{proof}

Given the mild assumptions of this section, an opfibred 1-cell $(K,F)$ has a
right adjoint in the 2-category $\Cat^{\mathbbm{2}}$ provided $K$ has a
right adjoint.

\begin{cor}
  \label{cor:3}
  Suppose that in an opfibred 1-cell $(K,F)\colon U\to V$ as in
  \eqref{eq:startingshape}
  \begin{enumerate}
  \item 
    $U$ and $V$ have initial objects preserved under fibre change.
  \item \label{item:13}
    The restriction $\ca{C}_X\to\ca{D}_{F(X)}$ of $K$ to each fibre has a right
    adjoint $N_X$.
  \item \label{item:14}
    The change-of-base functors for the opfibration $U$ have
    right adjoints.
  \end{enumerate}
  Then, $K$ has a right adjoint if and only if $(K,F)$ has a right adjoint
  $(R,G)$ in $\Cat^{\mathbbm{2}}$. Moreover, the restriction of $R$
  to each fibre is isomorphic to
  \begin{equation*}
    \label{eq:36}
    \ca{D}_Y\xrightarrow{\varepsilon_Y^*}\ca{D}_{FG(Y)}
    \xrightarrow{N_{F(Y)}}\ca{C}_{G(Y)}
  \end{equation*}
  where $\varepsilon_Y\colon FG(Y)\to Y$ is the counit of $F\dashv G$,
  $Y\in\caa{Y}$, and $(\varepsilon_Y)_!\dashv\varepsilon_Y^*$.
\end{cor}
\begin{proof}
  The `only if' part of the statement is trivial, so we only treat the `if'
  part. 
  The restriction of $K$ to fibres preserves initial objects, as a left
  adjoint. Assuming that $K$ has a right adjoint, then Theorem~\ref{thm:1}
  gives a right adjoint of $F$, which together with the hypotheses \ref{item:13}
  and \ref{item:14} yields the required adjunction in
  $\Cat^{\mathbbm{2}}$, by Corollary~\ref{cor:4}.
\end{proof}
It is worth noting that the corollary constructs from $K\dashv R$ an adjunction
$(K,F)\dashv (R',G)$ where $R'$ need not be equal to $R$ (though they are of course isomorphic).

\section{Measuring comodules}\label{sec:enrichmentModComod}

In this section we give a definition of measuring comodule using the language of
the global categories of (co)modules. Using that these categories are (op)fibred
over the
categories of (co)monoids, we look at their local presentability and at a
natural monoidal
structure they support.

\subsection{Global categories of modules and comodules}\label{globalcats}

We begin by recalling a category of (co)modules that appeared in
\cite{Grothendieckcategoriesfibrees} or \cite[Example 1.10]{Grayfibredandcofibred}.
Suppose $\ca{V}$ is a monoidal category.

\begin{defi}\label{defComod}
The \emph{global category of comodules} $\Comod(\ca{V})$
is the category of all right $C$-comodules $X$ for
any comonoid $C$. We often write $X_C$ for clarity.
A morphism $k_g\colon X_C\to Y_D$
for $X$ a $C$-comodule and $Y$ a $D$-comodule
is a pair $(k,g)$ consisting of a comonoid morphism $g\colon C\to D$
and an arrow $k\colon X\to Y$ in $\ca{V}$ which makes
the diagram
\begin{equation}\label{eq:comodmap}
  \begin{tikzcd}[row sep=small]
    X \arrow[r, "\delta"] \arrow[d, "k\,"'] &
    X \otimes C \arrow[r, "X\otimes g"] &
    X \otimes D \arrow[d, "\, k\otimes D"] \\
    Y \arrow[rr, "\delta"] &&
    Y \otimes D
  \end{tikzcd}
\end{equation}
commute\footnote{Notice that the diagram can be simplified as $\delta\circ k=(k\otimes g)\circ\delta$, but we here choose to emphasize the opfibration structure explained below, and write this as a $D$-comodule morphism.}. Dually, the \emph{global category
of modules} $\Mod(\ca{V})$ has as objects all right
$A$-modules $M$ for any monoid $A$, and
morphisms are $p_f\colon M_A\to N_B$
where $f\colon A\to B$ is a monoid morphism
and $p\colon M\to N$ makes the dual diagram commute.
We will drop the monoidal category $\ca{V}$ from the notation when there is no
room for ambiguity and write simply $\Comod$ and $\Mod$.
\end{defi}

There are obvious forgetful functors
\begin{equation}\label{forgetGV}
 V\colon \Mod(\ca{V})\longrightarrow\Mon(\ca{V})\quad\textrm{and}\quad
U\colon \Comod(\ca{V})\longrightarrow\Comon(\ca{V})
\end{equation}
which map any module $M_A$ (resp.~comodule $X_C$) to its
monoid $A$ (resp.~comonoid $C$). In fact,
$V$ is a split fibration and $U$ is a
split opfibration.
The chosen cartesian and cocartesian liftings are
\begin{gather}\label{canoncart}
\Cart(f,N)=(1_{N},f)\colon f^*N\to N\textrm{ in }\Mod(\ca{V}) \\
\Cocart(g,X)=(1_{X},g)\colon X\xrightarrow{}g_!X\textrm{ in }\Comod(\ca{V})\notag
\end{gather}
where the module $f^*N$ is $N$ with the $A$-module structure $N\otimes A\to
N\otimes B\to N$ given by precomposing its $B$-module structure with
$(1_N\otimes f)$, and $g_!X$ has a dual description -- namely the first line of (\ref{eq:comodmap}).

The fibre over a comonoid $C$ is the category of $C$-comodules $\Comod_C$, and for a monoid $A$
it is $\Mod_A$. A morphism in $p_f\colon M_A\to N_B$ in $\Mod$ can be equivalently
described as a morphism $M\to f^*N$ in $\Mod_A$.

The opfibration $\Comod$, regarded as an indexed category over cocommutative
$\ringk{}$-coalgebras, was studied in detail in \cite{PareGrun}.

\begin{rmk}
  \label{rmk:1}
  The change-of-base functor $f_!\colon\Comod_C\to\Comod_D$
  induced by the comonoid morphism $f\colon C\to D$ has a right adjoint when
  $\Comod_C$ has equalizers. This is a standard fact,
  easily derived from Dubuc's Adjoint Triangle
  Theorem~\cite{AdjointTriangles} (see also the proof of Dubuc's theorem
  in~\cite[Lemma~2.1]{zbMATH05695332}, where the existence of equalizers in
  $\mathcal{A}$ should be added to the hypotheses).
\end{rmk}

For future reference, we note the following.

\begin{rmk}
  \label{rmk:2}
  The categories of (co)modules exhibit functorial behavior with respect to
  (op)lax monoidal functors, as described below. Given a lax monoidal functor
  $S\colon\ca{V}\to\ca{W}$, there are naturally induced functors $\Mon(S)\colon \Mon(\ca{V})\to\Mon(\ca{W})$ and
  $\Mod(S)\colon \Mod(\ca{V})\to\Mod(\ca{W})$ that ensure the commutativity of
  the following diagram:
  \begin{equation*}
    \label{eq:42} \begin{tikzcd}[column sep=large]
      \Mod(\ca{V})\ar[r,"\Mod(S)"]\ar[d]&\Mod({\ca{W}})\ar[d]\\
      \ca{V}\times\Mon(\ca{V})\ar[r,"S\times\Mon(S)"]&\ca{W}\times\Mon(\ca{W})
    \end{tikzcd}
  \end{equation*}
  The functor $\Mod(S)$ operates by
  mapping a right $A$-module $M\otimes A\to M$ to the composition
  $S(M)\otimes S(A)\to S(M\otimes A)\to S(M)$. Similarly, if $\tau\colon
  S\Rightarrow T$ is a monoidal natural transformation between lax monoidal
  functors, then $\tau_M\colon S(M)\to T(M)$ is a morphism $S(M)_{S(A)}\to
  T(M)_{T(A)}$ in $\Mod(\ca{W})$.

  Additionally, $\Mod(S)$ strictly preserves the
  cleavage given by the cartesian liftings \eqref{canoncart}:
  $$\Mod(S)\Cart(f,N)=\Mod(S)(1_N,f)=(1_{S(N)},S(f))=\Cart(S(f),S(N)).$$

  In the case of braided monoidal categories, it is not hard to show that
  $\Mod(S)$ is a lax monoidal functor if $S$ is a braided lax monoidal
  functor. The proof consists in showing that the components $S(M_A)\otimes
  S(N_B)\to S((M\otimes N)_{A\otimes B})$ of the lax monoidal structure of $S$,
  are morphisms over the monoid morphism $S(A)\otimes S(B)\to S(A\otimes B)$.
\end{rmk}

The following result is mentioned in \cite[Thm.~45]{LinearReps} for
the particular case of $\ca{V}=\Mod_R$ for a commutative ring $R$.

\begin{prop}\label{Comodcomonadic}
The functor $F\colon \Comod\to\ca{V}
\times\Comon(\ca{V})$ which maps an object
$X_C$ to the pair $(X,C)$ is comonadic.
\end{prop}

\begin{proof}
Define a functor
$\ca{V}\times\Comon(\ca{V})\to\Comod$ sending an object $(X,D)$ to the
comodule $A\otimes D$ with $D$-comodule structure given by $1_A\otimes\Delta$,
where $\Delta$ is the comultiplication of $D$.
A morphism $(f,g)\colon (X,D)\to (Y,E)$ is sent to $(f\otimes g)_g$.
This establishes an adjunction $F\dashv R$.
The induced comonad $FR$ on $\ca{V}\times\Comon$,
given by $(X,D)\mapsto(X\otimes D,D)$, has $\Comod$ as its category of coalgebras.
\end{proof}

By the previous proposition, $\Comod(\ca{V})$ is cocomplete whenever $\ca{V}$ is
cocomplete.
Dually, $\Mod(\ca{V})$ is monadic over
the category $\ca{V}\times\Mon(\ca{V})$.
These facts are used to show the following.

\begin{prop}\label{ModComodlp}
If $\ca{V}$ is a locally presentable monoidal category and
$-\otimes-$ is accessible, then
$\Mod(\ca{V})$ and $\Comod(\ca{V})$ are locally presentable.
\end{prop}
\begin{proof}
  Monadicity (resp., comonadicity) of $\Mon(\ca{V})$ (resp., $\Comod(\ca{V})$) over
  $\ca{V}\times\Mon(\ca{V})$ (resp., $\ca{V}\times\Comon(\ca{V}))$, together
  with Theorem~\ref{Monadiccomonadicpresentability} and the comments above,
  yield the result.
\end{proof}

When $\ca{V}$ is braided monoidal, 
$\Comod(\ca{V})$ and $\Mod(\ca{V})$ are monoidal categories
as well: if $c$ is the braiding,
the object $X_C\otimes Y_D$ is a comodule over 
the comonoid $C\otimes D$ via the coaction
\begin{equation*}\label{coactioncomod}
X\otimes Y\xrightarrow{\delta_X\otimes\delta_Y}
X\otimes C\otimes Y\otimes D
\xrightarrow{1\otimes c_{C,Y}\otimes1} X\otimes Y\otimes C\otimes D. 
\end{equation*}
and similarly for $M_A\otimes N_B\in\Mod$. If $c$ is a symmetry, then $\Comod$ and $\Mod$ are symmetric as well.
Moreover, in that case the functors 
$V$ and $U$ of (\ref{forgetGV})
are braided strict monoidal.

As a first application of the general fibred adjunctions theory
of the previous section, we can deduce monoidal closedness of $\Comod(\ca{V})$
when $\ca{V}$ is locally presentable, braided and closed.

\begin{prop}\label{Comodclosed}
 If  $\ca{V}$ is a locally presentable braided
 monoidal closed category, the symmetric monoidal $\Comod(\ca{V})$
 is closed. Furthermore, the tensor-hom adjunction on $\Comod(\ca{V})$ is part of an
 adjunction in $\Cat^{\mathbbm{2}}$
 \begin{equation*}
   \label{eq:38}
   \begin{tikzcd}[column sep=0.8in]
     \Comod(\ca{V})
     \arrow[r, shift left=1.3ex, "{-\otimes X_C}"] \arrow[d,"U"']
     &
     \Comod(\ca{V}) \arrow[l, shift left=1.3ex, "{\overline{\HOM}(X_C,-)}"] \arrow[d, "U"] \\
     \Comon(\ca{V}) \arrow[r, shift left=1.3ex, "{-\otimes C}"] & \Comon(\ca{V}) \arrow[l, shift left=1.3ex, "{\HOM(C,-)}"] \\
     \arrow[phantom, from=1-1, to=1-2, "\bot"] \arrow[phantom, from=2-1, to=2-2, "\bot"]
   \end{tikzcd}
\end{equation*}
where $\HOM$ as in (\ref{def:HOM}) is the internal hom of $\Comon(\ca{V})$.
\end{prop}
Before giving the proof, we point out that
  the existence of the internal hom in $\Comod$ can be deduced from the local
  presentability of $\Comod$ in the following way which, however, does not
  describe it explicitly. If $X_C\in \Comod$, then $F(-\otimes X_C)= F(-)\otimes
  (X,C)$ where $F$ is the comonadic functor from
  Proposition~\ref{Comodcomonadic}. Then, $-\otimes X_C$ is a left adjoint by \cref{Kelly}, since
  $\Comod(\ca{V})$ is locally presentable by \cref{ModComodlp} and both $F$ and $-\otimes (X,C)$ are cocontinuous. The
  present proposition has the advantage of giving extra information.

\begin{proof}
First, we observe that there exists an opfibred 1-cell
\begin{equation}\label{otimesopfibred1cell}
\begin{tikzcd}[column sep=0.5in]
\Comod \arrow[r, "{(-\otimes X_C)}"] \arrow[d, "U"'] & 
\Comod \arrow[d, "U"] \\
\Comon \arrow[r, "{(-\otimes C)}"'] & 
\Comon
\end{tikzcd}
\end{equation}
Indeed, given a comonoid morphism $f\colon D\to E$ and a $D$-comodule $Y$,
$\Cocart(f\otimes 1_C,Y\otimes X)\colon Y\to (f\otimes 1_C)_!(Y\otimes X)$ is
equal to $\Cocart(f,Y)\otimes 1_{X_C}$. This means that the top horizontal
functor in~\eqref{otimesopfibred1cell} is a (strict) opfibred 1-cell of (split)
opfibrations.


By Proposition \ref{Comonclosed}, there is an adjunction
$(-\otimes C)\dashv\HOM(C,-)$
between the bases of (\ref{otimesopfibred1cell}).
Also, if $\varepsilon$ is its counit, the composite
\begin{equation*}
 \Comod_\ca{V}(\HOM(C,D))\xrightarrow{(-\otimes X_C)}
 \Comod_\ca{V}(\HOM(C,D)\otimes C)\xrightarrow{(\varepsilon_D)_!}
 \Comod_\ca{V}(D)
\end{equation*} 
has a right adjoint $\overline{\HOM}_D(X_C,-)$
by Theorem \ref{Kelly}. Indeed $\Comod_\ca{V}(\HOM(C,D))$
is locally presentable by Corollary~\ref{comodlocpresent},
reindexing functors preserve all colimits as discussed in \cref{ComonoidsComodules}, and the 
commutative diagram below
implies that $(-\otimes X_C)$
preserves all colimits, since the bottom arrow 
does (by monoidal closedness of $\ca{V}$),
and the vertical functors are comonadic.
\begin{equation*}
  \begin{tikzcd}[column sep = large]
    \Comod_{\ca{V}}(E) \arrow[r, "{-\otimes X_C}"] \arrow[d] & \Comod_{\ca{V}}(E\otimes C) \arrow[d] \\
    \ca{V} \arrow[r, "{-\otimes X}"] & \ca{V}
  \end{tikzcd}
\end{equation*}

By \cref{cor:4}, there is a functor
$\overline{\HOM}(X_C,-)\colon \Comod\to\Comod$, given on the fibres by
$\overline{\HOM}_D(X_C,-)\colon\Comod_{\ca{V}}(D)\to\Comod_{\ca{V}}(\HOM(C,D))$,
and
an adjunction  in $\Cat^\mathbbm{2}$ as in the statement.
\end{proof}


It is rare for the monoidal category of comodules to be closed when the category
of comonoids is not, as shown by the following proposition.

\begin{prop}
  Let $\ca{V}$ be a braided monoidal category with initial object. Then, if the
  monoidal category $\Comod(\ca{V})$ is left closed, then $\Comon(\ca{V})$ is
  left closed.
\end{prop}
\begin{proof}
  This is a straightforward application of Theorem~\ref{thm:1} to the case of an
  opfibred 1-cell $(-\otimes X_C,-\otimes C)$ as
  in~\eqref{otimesopfibred1cell}. The fibres of the opfibration
  $\Comod\to \Comon$ are the categories of comodules, all
  of which have initial objects preserved by change of fibre. The theorem then
  provides the right adjoint to $(-\otimes C)$.
\end{proof}

Turning to a different sort of hom-functor, if $\ca{V}$ is a braided monoidal closed category,
the induced
functors $[-,-]\colon \Comod_C^\op\times\Mod_A
\longrightarrow\Mod_{[C,A]}$ as in (\ref{defMod[]})
glue together into a functor
\begin{equation}\label{defbarH}
  [-,-]\colon \Comod^\op\times\Mod\longrightarrow \Mod
\end{equation}
that sends $(X_C,M_A)$ to the object $[X,M]$ equipped with the induced
action of the (convolution) monoid $[C,A]$.
In fact, applying Remark~\ref{rmk:2} to the lax monoidal functor
$[-,-]\colon\ca{V}^\op\times\ca{V}\to\ca{V}$, we have a 
fibred 1-cell
\begin{equation}\label{barHHfibred1cell}
\begin{tikzcd}[column sep=large]
\Comod^\op \times \Mod \ar[r, "{[-,-]}"] \ar[d, "U^\op \times V"'] & \Mod \ar[d, "V"] \\
\Comon^\op \times \Mon \ar[r, "{[-,-]}","(\ref{defMon[]})"'] & \Mon
\end{tikzcd}
\end{equation}

The commutativity of the square on the left-hand side below establishes that
$[-,N_B]^\op$  is cocontinuous when $\ca{V}$ is cocomplete: 
the comonadic functors at the left and right create all colimits
and both functors at the bottom
have right adjoints, $[-,N]^\op\dashv[-,N]$ for the internal hom in $\ca{V}$ and $[-,B]^\op\dashv P(-,B)$ by \cref{measuringcomonoidprop}.
Moreover,
the top horizontal functor is cocontinous on fibres, as the square on the
right-hand side commutes.
\begin{equation}\label{cocontbarHN}
\begin{tikzcd}[column sep=2.5cm]
  \Comod \ar[r, "{[-,N_B]}^\op"] \ar[d] &
  \Mod^\op \ar[d] \\
  \ca{V} \times \Comon \ar[r, "{[-,N]}^\op \times {[-,B]}^\op"] &
  \ca{V}^\op \times \Mon^\op
\end{tikzcd}
\qquad
\begin{tikzcd}
  \Comod_C\ar[r,"{[-,N_B]}"]\ar[d]&\Mod_{{[C,B]}}^\op\ar[d]\\
  \ca{V}\ar[r,"{[-,N]}"]&\ca{V}^\op
\end{tikzcd}
\end{equation}

\subsection{Measuring comodules}
\label{sec:measuring-comodules}

Let us start this section by recalling the definition of measuring comodule in
the context of vector spaces, due
to M.~Batcherlor~\cite{Batchelor}.
Suppose given a coalgebra $C$ together with a (right) $C$-comodule $X$, and a
pair of algebras $A$ and $B$ with respective (right) modules $M$ and $N$. For a
fixed measuring $\phi\colon A\otimes C\to B$, a linear map $\psi\colon M\otimes
X\to N$ is a \emph{measuring} if the following equality holds, for $m\in M$,
$a\in A$ and $x\in X$,
\begin{equation*}
  \label{eq:19}
  \psi((m\cdot a)\otimes x)=\sum \psi(m\otimes x_0)\cdot \phi(a\otimes x_1)
\end{equation*}
where the coaction $\chi\colon X\to X\otimes C$ is written in Sweedler's
notation as $\chi(x)=
\sum x_0\otimes x_1$.

The pair $(X,\psi)$ is called \emph{measuring comodule}. The equation above is
precisely the condition that makes the transpose map
$\bar{\psi}\colon M\to\Hom_\ringk{}(X,N)$ a morphism of $A$-modules.

The definition given above is equivalent to that in~\cite{Batchelor} except for the
fact that the latter expresses it in terms of maps $C\to \Hom_\ringk{}(A,B)$ and
$X\to\Hom_\ringk{}(M,N)$ and left (co)modules.

For its generalization, consider a \emph{braided} monoidal category
$\ca{V}$. For $M_A,N_B\in\Mod$ and $X_C\in\Comod$, we say that a pair of
morphisms $(\phi,\varphi)$, where
$\phi\colon A\otimes C\to B$ and $\varphi\colon M\otimes X\to N$ in $\ca{V}$, is
a \emph{module measuring morphism} from $M$ to $N$ if $\phi$ is a monoid measuring
morphism (\ref{measuring_mon}) and the
following diagram in $\mathcal{V}$ commutes, where $c$ denotes the braiding,
$\mu$ the respective module actions, and $\chi$ the coaction.
\begin{equation}
  \label{eq:18}
  \begin{tikzcd}[column sep=.7in]
    M\otimes{}A\otimes{}X \ar[r, "M\otimes{}A\otimes{}\chi"]
    \ar[d, "\mu\otimes{}X"'] & M\otimes{}A\otimes{}X\otimes{}C
    \ar[r, "M\otimes{}c_{A,X}\otimes{}C"] & M\otimes{}X\otimes{}A\otimes{}C \ar[d, "\varphi\otimes{}\phi"] \\
M\otimes{}X \ar[r, "\varphi"] & N & N\otimes{}B \ar[l, "\mu"']
\end{tikzcd}
\end{equation}

There are of course similar definitions for all the combinations of left or right
comodules and left or right modules. One of the reasons we have decided to work
with comodules and modules on the same side, right side in our case, is that the
the diagram~\eqref{eq:18} requires only one instance of the braiding.

\begin{ex}
  In the case $C=I=X$, $\psi\colon M\to N$ is just a morphism in \Mod\ over
  $\phi$. More interesting examples will be given in
  Sections~\ref{sec:deriv-meas-comod} and \ref{sec:higher-derivations}.
\end{ex}

\begin{lem}
  \label{l:1}
  When the braided monoidal category $\ca{V}$ is closed, module measurings as
  described above are in bijection with morphisms $M_A\to [X_C,N_B]_{[C,B]}$ in
  $\Mod$.  
\end{lem}
\begin{proof}
  A morphism $M_A\to [X_C,N_B]_{[C,B]}$ consists of a monoid morphism
  $\hat\phi \colon A\to[C,B]$ together with a $A$-module morphism $\hat\varphi\colon
  M\to\hat\phi^*[X,N]$ (see the comments before Remark~\ref{rmk:1}).
  These are in bijection with measurings $\phi\colon
  A\otimes C\to B$ together with morphisms $\varphi\colon M\otimes X\to N$ that
  make \eqref{eq:18} commutative.
\end{proof}

Given comodules $X_C$, $ X'_{C'}$ and modules $M_A$, $N_B$,
$ M'_{ A'}$, $ N'_{ B'}$ in a symmetric monoidal category and
measurings $(\psi,\phi)\colon X\otimes M\to N$ and
$(\psi',\phi')\colon X'\otimes M'\to N'$, we saw in Lemma~\ref{l:7}
how to construct a monoid measuring $\phi\bullet\phi'$ making $C\otimes
C'$ into a measuring comonoid from $A\otimes A'$ to $B\otimes B'$. In a similar
fashion we may consider
\begin{equation*}
  \psi\bullet\psi'\colon
  X\otimes X'\otimes M\otimes M'
  \cong
  X\otimes M\otimes X'\otimes M'
  \xrightarrow{\psi\otimes\psi'}N\otimes N'
\end{equation*}
where the isomorphism is the one induced by the braiding.
\begin{lem}
  \label{l:6}
In a symmetric monoidal category, the morphism $\psi\bullet\psi'$ is a module measuring.
\end{lem}
The proof of the lemma is straightforward. In the case of
a symmetric monoidal \emph{closed} category, the proof can be written in a few
lines. The transpose of the morphism $\psi\bullet\psi'$ under the adjunction
between $(-\otimes X\otimes X')$ and $[X\otimes X',-]$ in $\ca{V}$ is none other
than
\begin{equation*}
  \label{eq:50}
  M\otimes M'\xrightarrow{\hat\psi\otimes\hat\psi'}[X,N]\otimes[X',N']
  \to[X\otimes X',N\otimes N']
\end{equation*}
where the last morphism comes from the braided monoidal structure on $[-,-]$
induced by the symmetry, and, therefore, it is a module morphism.

\section{Universal measuring comodules}
\label{Universalmeasuringcomodule}
We begin this section by introducing universal measuring comodules in a monoidal category, and
comparing our definition with the original one in vector spaces from M. Batchelor~\cite{Batchelor} . Next, we
interpret module derivations through the lens of measuring comodules, laying the
groundwork for further exploration in Section~\ref{sec:higher-derivations}. We
then demonstrate that the existence of a universal measuring comodule for the
module pair $M_A$, $N_B$ implies the existence of a universal measuring comonoid
for $A$, $B$, requiring only a minimal condition on the base monoidal
category. Finally, we construct universal measuring comodules as adjoints to a
fibred 1-cell, leveraging results from Section~\ref{sec:fibrations}.

\subsection{The definition}\label{sec:2}

If $M_A$ and $N_B$ are modules,
\cite{Batchelor} defined a comodule which is universal amongst $C$-comodules for a
fixed measuring coalgebra from $A$ to $B$. In contrast, we leave the measuring
comonoid $C$ to vary freely. This is achieved by using the global categories of
(co)modules.

We start with a braided monoidal closed categoy $\ca{V}$.
\begin{defi}
\label{df:2}
We define an object $Q(M,N)$ in $\Comod(\ca{V})$,
the \emph{universal measuring comodule}, by an isomorphism
\begin{equation*}
\Comod(\ca{V})(X,Q(M,N))\cong\Mod(\ca{V})(M,[X,N])
\end{equation*}
natural in $X$, where $[X,N]$ is as in (\ref{defbarH}).
In other words, $Q(M,N)$ is a representing object of the functor assigning to
each comodule $X$ the set of module measurings $X\otimes M\to N$; see Lemma~\ref{l:1}.
\end{defi}

Let us compare Definition~\ref{df:2} with a straightforward translation of Batchelor's
definition~\cite{Batchelor} to a braided monoidal closed category. 
Given monoids $A$ and $B$ and a measuring comonoid $(C,\phi)$ with
$\phi\colon A\otimes C\to B$, denote by $\hat\phi\colon A\to [C,B]$ the
associated monoid morphism (\cref{unimeascom}). According to~\cite{Batchelor}, a \emph{universal $(C,\phi)$-measuring comodule} for the
$A$-module $M$ and the $B$-module $N$ is a measuring comodule
$(Q^\phi(M,N),\psi)$ that represents the functor
$\Mod_A(M,\hat\phi^*[-,N])$ from $\Comod_C^{\mathrm{op}}$ to
$\mathsf{Set}$,
\begin{equation*}
  \Comod_C(X,Q^\phi(M,N))\cong\Mod_A(M,\hat\phi^*[X,N])
\end{equation*}
In other words, it is a representation of the presheaf
$$X\mapsto\{\psi\colon
M\otimes X\to N: (\psi,\phi)\text{ is a module measuring}\}.$$
\begin{lem}
  \label{l:4}
  Let $M_A$ and $N_B$ be modules. If the universal measuring comonoid $\mu\colon
  A\otimes P(A,B)\to B$ exists, then $Q^\mu(M,N)\cong Q(M,N)$, one side existing
  if and only if the other does.
\end{lem}
\begin{proof}
  We apply \cref{thm:3} to the following opfibred 1-cell.
  \begin{equation}\label{eq:40}
    \begin{tikzcd}[column sep=.5in]
      \Comod \ar[r, "{[-,N_B]}^\op"] \ar[d] &
      \Mod^\op \ar[d] \\
      \Comon(\ca{V}) \ar[r, "{[-,B]}^\op"] &
      \Mon(\ca{V})^\op
    \end{tikzcd}
  \end{equation}
  The universal measuring comonoid $P(A,B)$ is a coreflection of $A$ along the
  functor at the bottom, with counit that, as a morphism in
  $\Mon(\ca{V})$, is the transpose $\hat\mu\colon A\to[P(A,B),B]$ of $\mu$.
  Given an $A$-module $M$, the universal meausiring
  comodule $Q(M,N)$ is precisely a coreflection of $M$ along
  $[-,N]^\op$. Similarly, $Q^\mu(M,N)$ is a coreflection of $M$ along
  $\mu^*[-,N]^\op\colon\Comod(P(A,B))\to \Mod(A)^{\op}$. The result now follows from \cref{thm:3}.
\end{proof}

\subsection{Derivations and measuring comodules}
\label{sec:deriv-meas-comod}

Prior to further exploring the abstract elements of measuring comodules, we
pause to discuss their connection to module derivations. To do this, let us first
revisit some key definitions.

Consider a $\ringk{}$-algebra morphism $f\colon A\to B$ and an $f$-derivation $\delta$,
such that $(\delta,f)\in\operatorname{Der}(A,B)$ as defined in \cref{unimeascom}. For any $M_A$ and $N_B$, we
can examine pairs $(h,D)$, where $h\colon M_A\to N_B$ is a module map over $f$,
and $D\colon M\to N$ is a linear map that satisfies
\begin{equation*} D(x\cdot a)=h(x)\cdot\delta(a)+D(x)\cdot f(a)\qquad x\in M,\
a\in A.
\end{equation*}
These
morphisms, which we refer to as module $h$-derivations, are commonly found in
the literature when $h$ is the identity, simplifying their defining identity to
$D(m\cdot a)=m\cdot\delta(a)+D(m)\cdot a$. We denote the set of all module
derivations by $\operatorname{MDer}(M_A,N_B)$, which comes with an obvious
projection function into $\operatorname{Der}(A,B)$.

\begin{ex}
  \label{ex:5}
  In this example we demonstrate that $Q(M_A,N_A)$ encapsulates all
  the details about module derivations, expanding on Example~\ref{ex:3}. We
  start with the coalgebra $C_1=\ringk{}\cdot g\oplus \ringk \cdot v$, where $g$ is a group-like
  element and $v$ is a $g$-primitive element. Viewing this coalgebra as a comodule
  over itself, we will show that
  \begin{equation*}
    \Comod(C_1,Q(M_A,N_A))\cong
    \operatorname{MDer}(M_A,N_A).
  \end{equation*}
  The left-hand side is naturally in one-to-one correspondence with module
  measurings $(\psi,\phi)\colon N_A\otimes (\ringk{}\cdot g\oplus \ringk{}\cdot v)\to M_B$. As
  observed in Example~\ref{ex:3}, $\phi$ corresponds to an element of
  $\operatorname{Der}(A,B)$, specifically the algebra map $\phi(-\otimes g)$ and
  the $\phi(-\otimes g)$-derivation $\phi(-\otimes v)$. The morphism
  $\psi(-\otimes g)$ is a module morphism $M_A\to N_B$ over the algebra map
  $\phi(-\otimes g)$. Additionally, the morphism $\psi(-\otimes v)$ satisfies
  \begin{equation*} \psi((x\cdot a)\otimes v)= \psi(x\otimes
    g)\cdot\phi(a\otimes v)+\psi(x\otimes v)\cdot\phi(a\otimes g) \qquad x\in M,\
    a\in A.
  \end{equation*}
  Therefore the pair $\psi(-\otimes v)$, $\psi(-\otimes g)$ together with
  $(\phi(-\otimes v),\phi(-\otimes g))\in\operatorname{Der}(A,B)$ constitute an
  element of $\operatorname{MDer}(M_A,N_B)$.
\end{ex}

\subsection{Measuring coalgebras from measuring comodules}
\label{sec:relat-betw-exist}
This section shows $Q(M_A,N_B)$ is a $P(A,B)$-comodule and, in particular, that this
universal measuring comonoid exists. In order to do so, we only require that
$\ca{V}$ should have an initial object.

\begin{lem}
  \label{l:2}
  Suppose that the braided monoidal closed category $\ca{V}$ has an initial
  object. If the universal measuring comodule $Q(M,N)$ exists for an $A$-module
  $M$ and a $B$-module $N$, then the universal measuring comonoid $P(A,B)$
  exists. Furthermore, given a choice of $Q(M,N)$, its underlying comonoid
  satisfies the universal property of $P(A,B)$.
\end{lem}
\begin{proof}
  We begin by employing Lemma~\ref{l:squareisexact} to verify that
  \eqref{eq:40} is an exact square.
  Both vertical arrows in the square are opfibrations with initial objects
  preserved under fibre change.
  The opfibred 1-cell in the diagram preserves initial objects: the initial object in
  $\Comod_C$ is the initial object $0\in \ca{V}$ with its unique $C$-comodule
  structure, and $[0,N]\cong 1$ is initial in $\Mod_{[C,B]}^\op$.

  Lemma~\ref{l:extsq-adj} completes the proof, since $Q(M,N)$ is a coreflection
  of $N$ along the top of the square~\eqref{eq:40}, while $P(A,B)$ is a
  coreflection of $V(M)=A$ along the bottom of the same square. The lemma tells
  us that this coreflection can be constructed as $U(Q(M,N))$, in other words,
  the underlying comonoid of $Q(M,N)$.
\end{proof}

\subsection{The universal measuring comodule as an adjoint functor}
\label{sec:univ-meas-comod}

The present section applies the adjointness results from Section~\ref{sec:1} to
prove the existence of the universal measuring comodule.
Our central hypothesis will be the local presentability of the base braided
monoidal closed category. The existence of a right adjoint $Q(-,N_B)$ for
$[-,N_B]$ can easily be derived from the cocontinuity of the latter functor (see
the text above~\eqref{cocontbarHN}) and the local presentability of $\Comod$
(Proposition~\ref{ModComodlp}). Since the comonoid part of $Q(M_A,N_B)$ is
isomorphic to any choice of universal measuring comonoid $P(A,B)$
(Lemma~\ref{l:2}), naturally in $M_A$, there is a functor isomorphic to
$Q(-,N_B)$ that lies above $P(-,B)$. The following proposition, however,
constructs $Q$ in a way that provides us with more information.

\begin{prop}\label{propmeasuringcomodule}
  Let $\ca{V}$ be a locally presentable braided monoidal closed
  category. Then, \eqref{eq:40} has a right adjoint in the 2-category
  $\Cat^{\mathbbm{2}}$, given by $Q(-,N_B)$ on the total categories and by
  $P(-,B)$ on the base categories.
  \begin{displaymath}
    \begin{tikzcd}[row sep=0.5in, column sep=0.8in]
      \Comod \arrow[r, "{[-,N_B]}^\op", shift left=1.5ex] \arrow[d] 
      \arrow[r, phantom, "\bot"] &
      \Mod^\op \arrow[l, "{Q(-,N_B)}", shift left=1.5ex] 
      \arrow[d] \\
      \Comon \arrow[r, "{[-,B]}^\op", shift left=1.5ex] 
      \arrow[r, phantom, "\bot"] &
      \Mon^\op \arrow[l, "{P(-,B)}", shift left=1.5ex]
    \end{tikzcd}
  \end{displaymath}
  In particular, $Q(M_A,N_B)$ is a
  $P(A,B)$-comodule and the comonoid component of the universal module measuring
  $M\otimes Q(M_A,N_B)\to N$ is the universal measuring $A\otimes P(A,B)\to B$.
\end{prop}
\begin{proof}
  The proof is an application of Corollary~\ref{cor:4} to the opfibred
  1-cell~\eqref{eq:40}.
  First of all, Theorem~\ref{measuringcomonoidprop} gives a right adjoint $P(-,B)$ to the
  bottom functor $[-,B]^\op$.
  Moreover, the composite functor
\begin{equation*}
\Comod_\ca{V}(P(A,B))\xrightarrow{[-,N_B]^\op}
\Mod_\ca{V}^\op([P(A,B),B])\xrightarrow{(\varepsilon_A)_!}
\Mod_\ca{V}^\op(A)
\end{equation*}
where $\varepsilon_A\colon [P(A,B),B]\to A$ in $\Mon(\ca{V})^\op$ is the counit
of $[-,B]^\op\dashv P(-,B)$, has a right adjoint by Theorem~\ref{Kelly}:
$\Comod_\ca{V}(C)$ is a locally presentable category by
Corollary~\ref{comodlocpresent}, the reindexing functors are always
cocontinuous as seen in Section~\ref{ComonoidsComodules},
and the fibrewise $[-,N_B]^\op$ is also cocontinuous, as remarked above (\ref{cocontbarHN}).
We therefore obtain a right adjoint
\begin{equation*}
 Q(-,N_B)\colon \Mod_\ca{V}(A)^\op\longrightarrow
\Comod_\ca{V}(P(A,B))
\end{equation*}
and an adjunction in the 2-category $\Cat^{\mathbbm{2}}$ as depicted in
the statement.
\end{proof}

We have an induced functor of two variables
\begin{equation*}
 Q(-,-)\colon \Mod^\op\times\Mod\longrightarrow\Comod
\end{equation*}
called the \emph{universal measuring comodule functor},
which is the parametrized adjoint of $[-,-]^\op$.

\begin{lem}\label{hello2}
Suppose $A$ and $B$ are monoids in $\ca{V}$ regarded as regular
modules over themselves. Then there are natural
isomorphisms of $P(A,B)$-comodules
\begin{equation*}
[V,N] \otimes P(A,B)\cong Q(V\otimes A,N)
\end{equation*}
for any object $V$ in $\ca{V}$ and
$B$-module $N$.
In particular,
${A^{\circ}}_{A^{\circ}}\cong Q(A,I)_{A^{\circ}}$,
where $A^{\circ}=P(A,I)$ is the Sweedler dual comonoid.
\end{lem}

\begin{proof}
The diagram of the left adjoints below commutes,
as we already saw in~\eqref{cocontbarHN}.
\begin{equation*}\label{bigdiagadj}
  \begin{tikzcd}[row sep=1.1cm, column sep=1.3in]
    \Comod \arrow[r, "{[-,N_B]}^\mathrm{op}", shift left=2ex]
    \arrow[r, phantom, "\bot"]
    \arrow[d, shift right=2ex]
    \arrow[d, phantom, "\dashv"] &
    \Mod^\mathrm{op} \arrow[l, "{Q(-,N_B)}", shift left=2ex]
    \arrow[d, shift right=2ex]
    \arrow[d, phantom, "\dashv"] \\
    \ca{V}\times\Comon
    \arrow[r, "{[-,N]}^\mathrm{op}\times {[-,B]}^\op", shift left=2ex]
    \arrow[r, phantom, "\bot"]
    \arrow[u, shift right=2ex] &
    \ca{V}^\op\times \Mon^\op
    \arrow[l, "{[-,N]}\times {P(-,B)}", shift left=2ex]
    \arrow[u, shift right=2ex]
  \end{tikzcd}
\end{equation*}
Therefore the corresponding square of right adjoints
commutes up to isomorphism. Given a monoid
$A$ and an object $V$, the right-left composition of right adjoints has the
effect $(V,A)\mapsto(V\otimes A,A)\mapsto Q((V\otimes A)_A,N_B)$. On the other
hand, the left-top composition does $(V,A)\mapsto([V,N],P(A,B))\mapsto
[V,N]\otimes P(A,B)$. This yields the natural isomorphism of the satement.
In the case of $V=N=B=I$ we get
the particular case of the Sweedler dual.
\end{proof}

\section{Enrichment of modules in comodules}\label{sec:enrichmentofmodincomod}

Similarly to how Theorem~\ref{monoidenrichment} established the enrichment of
monoids in comonoids in \cite[\S 5]{Measuringcomonoid}, we will now use
universal measuring comodules and the theory of actions of monoidal categories
to enrich the global category of modules in the global category of
comodules.

In the following commutative diagram, the functor at the bottom is an action of
$\ca{V}\times\Comon$ on $\ca{V}^\op\times\Mon^\op$. By
restricting this action along the
strict monoidal functor $\Comod\to\ca{V}\times\Comon$, we obtain an
action of $\Comod$ on $\ca{V}^\op\times\Mon^\op$.
\begin{equation}\label{barHopaction}
  \begin{tikzcd}[row sep=small,column sep=70pt]
    \Comod\times\Mod^\op\ar[r,"{[-,-]^\op}"]\ar[d]&\Mod^\op\ar[d]\\
    \ca{V}\times\ca{V}^\op \times\Comon\times\Mon^\op
    \ar[r,"{[-,-]^\op\times[-,-]^\op}"]&
    \ca{V}^\op\times\Mon^\op
  \end{tikzcd}
\end{equation}

\begin{prop}
  \label{prop:1}
Let $\ca{V}$ be a braided monoidal closed category. The monoidal category
$\Comod$ acts on $\Mod^\op$ via the functor at the top of the diagram above.
This action is strictly preserved by the forgetful functor on the right.
Moreover, this action is opmonoidal if $\ca{V}$ is symmetric.
\end{prop}
\begin{proof}
  The proof is straightforward and based on Remark~\ref{rmk:2} and the fact that
  the braiding endows both $\otimes$ and $[-,-]$ with lax monoidal structures
  such that the canonical isomorphism
  $\sigma_{X,Y,Z}\colon [X\otimes Y,Z]\cong[X,[Y,Z]]$ is monoidally natural.

  We can regard comodules $X_C$ and $Y_D$, and a module $M_A$, as a single module
  $(X,Y,M)_{(C,D,A)}$ in the monoidal category
  $(\ca{V}^\op)^2\times\ca{V}$. Applying the lax monoidal functors
  $[-\otimes-,-]$ and $[-[-,-]]$ we obtain modules $[X\otimes Y,M]_{[C\otimes
    D,A]}$ and $[X,[Y,M]]_{[C,[D,A]]}$, which are the domain and codomain of a
  morphism $\sigma_{X,Y,M}$ in $\Mod$.

  In a similar fashion, $[I,-]\cong 1_{\ca{V}}$ is a monoidal natural
  transformation, which gives a natural transformation $[I,M_A]\cong M_A$ in
  $\Mod$.

  It remains to be shown that these two natural isomorphisms satisfy the axioms
  for an action of the monoidal category $\Comod^{\op}$ on $\Mod$. Fortunately,
  the diagrams required to commute in $\Mod$ do so after applying the forgetful
  functor to $\ca{V}\times\Mon$, sparing us further effort.

  Finally, $[-,-]\colon \ca{V}\times\ca{V}^\op\to\ca{V}^\op$ is a \emph{braided} opmonoidal
  functor when $\ca{V}$ is symmetric.
  The dual of the last paragraph of Remark~\ref{rmk:2} tells us that the functor
  induced by $[-,-]$ on categories of comodules is oplax monoidal, with a
  structure preserved by the forgetful functor on the right of~\eqref{barHopaction}.
\end{proof}

The theory of enrichment induced by actions, as outlined in
\cref{sec:actionenrich}, can now be applied to yield the following
outcome.

\begin{thm}\label{ModenrichedinComod}
Let $\ca{V}$ be a locally  presentable
symmetric monoidal closed category.
Then $\Mod$ is a tensored and cotensored $\Comod$-enriched
symmetric monoidal
category $\underline{\Mod}$ with hom-objects $\underline{\Mod}(M_A,N_B)$ given by $Q(M,N)_{P(A,B)}$
and cotensor products of $N_B$ by $X_C$ given by $[X,N]_{[C,B]}$.
\end{thm}
\begin{proof}
  First, we can deduce that there is a $\Comod$-category, say $\mathcal{M}$,
  with underlying category $\Mod^{\op}$. Indeed, $\Comod$ acts on $\Mod^\op$ via $[-,-]$, as
  seen in Proposition~\ref{prop:1}. Furthermore, each functor $[-,N_B]$ from
  $\Comod$ to $\Mod^\op$ has a
  right adjoint, by \cref{propmeasuringcomodule}, yielding an enriched
  category $\mathcal{M}$ with underlying category $\Mod^\op$, enriched homs
  $\mathcal{M}(M_A,N_B)=Q(N_B,M_A)$ and tensor product of $X_C$ by $M_A$ given
  by $[X_C,M_A]$; see Theorem~\ref{actionenrich}.
  Then, $\mathcal{M}^\op$ is the sought for enrichment of $\Mod$ to a
  $\Comod$-category. Its contensor products are the tensor products of
  $\mathcal{M}$.
  It only remains to deal with the assertion about tensor products.

  We want to show that $\mathcal{M}^\op$ has tensor products, which is to say
  that $\mathcal{M}$ has cotensor products. For this to be the case, it is
  enough for each action endofunctor $[X_C,-]$ of $\Mod^\op$ to have a right
  adjoint, by Theorem~\ref{actionenrich}.

  First recall that as mentioned
  in~\eqref{barHHfibred1cell}, there is a fibred 1-cell
  \begin{displaymath}
    \begin{tikzcd}[row sep=15pt]
      \Mod\ar[d]\ar[r,"{[X_C,-]}"]&\Mod\ar[d]\\
      \Mon(\ca{V})\ar[r,"{[C,-]}"]&\Mon(\ca{V})
    \end{tikzcd}
  \end{displaymath}
  We know that the bottom functor has a left adjoint $C\triangleright-$ (\cref{monoidenrichment}), with
  unit that we will denote by $\eta$. Therefore by the dual of \cref{cor:4}, the top functor of   
  the diagram has a left adjoint if, for any monoid $A$, the top composition in the
  diagram below has a left adjoint:
  \begin{displaymath}
    \begin{tikzcd}[column sep=.5in,row sep=small]
      \Mod({C\triangleright A})\ar[rr,"{[X_C,-]_{C\triangleright A}}"]\ar[d] &&
      \Mod({[C,C\triangleright A]})
      \ar[r,"{\eta^*_A}"]\ar[d] & \Mod(A)\ar[dl] \\
      \ca{V}\ar[rr,"{[X,-]}"] && \ca{V} &
    \end{tikzcd}
  \end{displaymath}
  The existence of the left adjoint $X_C\triangleright-$ of the top composition
  follows from the Adjoint Triangle Theorem, since the vertical and diagonal
  functors are monadic, and the bottom functor is a right adjoint.
  We thus obtain an adjunction in
  $\Cat^{\mathbbm{2}}$, and in particular $X_C\triangleright M_A$ is a
  module over $C\triangleright A$.
  The assertion about the enriched symmetric monoidality of $\Mod$ follows from
  Theorem~\ref{opmonactionmonenrich} and Proposition~\ref{prop:1}.
\end{proof}

The construction in the proof gives us a functor on the top of the following
diagram, lifting the Sweedler product (i.e. the tensor product of $\Mon$).
\begin{displaymath}
  \begin{tikzcd}
    \Comod\times\Mod\ar[r,"\triangleright"]\ar[d]&\Mod\ar[d]\\
    \Comon\times\Mon\ar[r,"\triangleright"]&\Mon
  \end{tikzcd}
\end{displaymath}

\begin{rmk}
It is the case that $\Mod\to\Mon$ is an \emph{enriched fibration} over the \emph{monoidal opfibration} $\Comod\to\Comon$; we here omit the details regarding such structures and we refer the interested reader to \cite[\S~4.1]{EnrichedFibration}.
\end{rmk}

\begin{cor}
  In the situation of the previous theorem,
  the universal measuring comodule functor
  $Q\colon \Mod^\mathrm{op}\times\Mod\to\Comod$ is a braided lax monoidal
  functor.
\end{cor}
\begin{proof}
  If $\ca{A}$ is
  a monoidal $\ca{C}$-category, where $\ca{C}$ is a (braided) monoidal category,
  then the hom functor $\ca{A}(-,-)\colon\ca{A}_\circ^\op\times\ca{A}_\circ\to\ca{C}$
  has a canonical (braided) lax monoidal structure given by the effect of the
  tensor product on homs: $\ca{A}(A,B)\otimes\ca{A}(A',B')\to\ca{A}(A\otimes
  A',B\otimes B')$. In the braided case, $\ca{A}^\op$ is equipped with the
  braiding given by the inverse of the braiding of $\ca{A}$.
\end{proof}

\begin{ex}
  \label{ex:4}
  The enrichment of modules in comodules described above induces, via changing
  the base of enrichment, an ordinary category whose morphisms $M_A\to N_B$ are
  module derivations. We refer to Example~\ref{ex:5}, where we exhibited an
  isomorphism between
  $\Comod(C_1,Q(M,N))$ and $\operatorname{MDer}(M,N)$.
  The right hand side of the isomorphism can be made into
  a functor $\Mod^\op\times\Mod\to\mathsf{Set}$ in an obvious way that makes the
  isomorphisms into a natural transformation.

  The coalgebra $C_1=\ringk{}\cdot g\oplus \ringk{}\cdot v$, where $g$ and $v$ are,
  respectively, a group-like and a $g$-primitive element, possesses the property of
  inducing a comonoid in $\Comod$. To see this, first note that $C_1$ is
  cocommutative, which implies that it gives rise to a cocommutative comonoid in
  $\Coalg{}$.  Moreover, the comultiplication
  $\Delta\colon C_1\to C_1\otimes C_1$ is a comodule morphism over the coalgebra
  morphism $\Delta$, and the counit $\varepsilon\colon C_1\to \ringk{}$ is a comodule
  morphism over $\varepsilon$. Thus, $C_1$ is a
  cocommutative comonoid in $\Comod$, which makes
  $\Comod(C_1,-)\colon\Comod\to\mathsf{Set}$ a braided lax monoidal functor. By
  change of base of enrichment along this functor we obtain a
  $\mathsf{Set}$-category whose objects are modules and whose hom-set from $M$
  to $N$ is $\operatorname{MDer}(M,N)$.
  For a generalisation of this example see Section~\ref{sec:higher-derivations}.
\end{ex}

\section{The coinvariants of the universal measuring comodule}
\label{sec:coinv-univ-meas}

This section explains the construction of coinvariants for the universal
measuring comodule associated with a pair of modules.
Coinvariants are central in the theory of quotients of algebraic
groups, where, if a quotient of an affine algebraic group $G$ by the action of a closed subgroup $H$
exists, then its algebra of functions $\mathcal O(G/H)$ is the space of coinvariants
of the $\mathcal O(H)$-comodule $\mathcal O(G)$ (see \cite{zbMATH06713849} Prop.~B.28 and proof of
step 2 in p.~598 for a more general statement). More generally, coinvariants
are central in the theory of Hopf algebras
\cite{zbMATH05906466,zbMATH00482792}.

Assume that $u\colon I \to C$ is a comonoid morphism within the category
$\ca{V}$. When $\ca{V}$ is the category of vector spaces, this morphism
corresponds to a group-like element in $C$. For the purposes of this section,
the braided monoidal closed category $\ca{V}$ must, at a minimum, possess
equalizers for coreflexive pairs.

If $X\in\Comod_C$, the object of coinvariants of $X$ is an object of $\ca{V}$ that we denote by $u^*X$: it represents the functor
$\Comod_C(u_!(-),X)$ from $\ca{V}^{\mathrm{op}}\cong\Comod_I^{\mathrm{op}}$ to
$\mathsf{Set}$. The existence of $u^*X$ is equivalent to the
existence of the equalizer of the coaction $X\to X\otimes C$ and $1_X\otimes u$,
with $u^*X$ being precisely this equalizer. The existence of coinvariants for
each $C$-comodule is equivalent to the existence of a right adjoint
$u^*\colon\Comod_C\to\Comod_I\cong\ca{V}$ to the change of fibre functor $u_!$.

Building towards the following proposition, it is well-known that $\Mod_A$ has a canonical
structure of a $\ca{V}$-category; this can be easily verified, or one may appeal to the fact that $\Mod_A$ is the
category of Eilenberg-Moore algebras for a $\ca{V}$-enriched monad;
see~\cite{zbMATH03264971}, \cite{zbMATH03360361} and
\cite[Thm.~15]{FormalTheoryMonadsI}. The enriched hom objects can be constructed
by means of coreflexive equalizers. One way of characterizing, up to isomorphism,
the $\ca{V}$-enriched hom from $M$ to $N$ is as a representing object of
the presheaf $\Mod_A(M,[-,N])\colon\ca{V}^{\mathrm{op}}\to\mathsf{Set}$, where
the $A$-module structure on $[Z,N]$ is induced by that of $N$, for $Z\in\ca{V}$.

Before presenting our result, we must establish a piece of
notation. Assume that the universal measuring comonoid $P(A,B)$ for the pair of
monoids $A$ and $B$ exists. Denote by $\bar{f}\colon I \to P(A,B)$ the comonoid
morphism that corresponds to the monoid morphism $f\colon A \to B$, as
determined by the universal property of $P(A,B)$; see~\eqref{eq:12}.

\begin{prop}
  \label{prop:4}
  With the above notation, if $Q(M,N)$ exists for a pair of modules $M_A$ and
  $N_B$, then the object of coinvariants $\bar{f} ^* Q(M,N)$ is isomorphic to the 
  $\ca{V}$-enriched internal hom from $M$ to $f^*N$ of the $\ca{V}$-category
  $\Mod_A$.
\end{prop}
\begin{proof}
  By the comments above this proposition, we are to show that $\bar{f}^*Q(M,N)$
  represents the functor $\Mod_A(M,[-,f^*N])$.
  There is a bijection, natural in $Z\in\ca{V}$, between morphisms $Z\to \bar{f}
  ^*Q(M,N)$ and morphisms $Z\to Q(M,N)$ in $\Comod$ that lie over the
  morphism $\bar{f} \colon I\to P(A,B)$ in $\Comon(\ca{V})$. These
  morphisms are in natural bijection with morphisms $M\to [Z,N]$ in $\Mod$ that
  lie over $f\colon A\to B$ in $\Mon(\ca{V})$. We have, then, a natural
  bijection $\ca{V}(Z,\bar{f}^* Q(M,N))\cong\Mod_A(M,[Z,f^*N])$ as
  required.
\end{proof}

\section{Higher derivations}
\label{sec:higher-derivations}

In this section we explain how higher derivations of $\ringk{}$-algebras and of modules
are particular features of the enrichment of the category of $\ringk{}$-algebras over
the category of $\ringk{}$-coalgebras and the category of modules over comodules, respectively. We begin by giving
the definition of higher derivations, due to Hasse and
Schmidt~\cite{zbMATH03027033}. Modern references include
\cite{zbMATH00043569} and \cite{zbMATH05263295}. Even though these and most
references work with commutative rings, the definitions carry over to our
non-commutative setting. We refer the reader to the survey \cite{zbMATH05919911}.
We define the non-commutative version of the Hasse-Schmidt algebra, absent from
the literature as far as we know, by means of the Sweedler product, and show
its relationship to the bimodule of (K\"alher) differentials~\cite[Prop.~III.10.17]{zbMATH03440338}. We
conclude by applying our machinery to higher derivations of modules, first
introduced in~\cite{zbMATH03843957}.

\subsection{Higher derivations of algebras}\label{sec:higherdevs}
Let $\ringk{}$ be a commutative ring, and let $A$ and $B$ be $\ringk{}$-algebras. A
\emph{higher derivation}, or \emph{Hasse-Schmidt derivation, of length
$0\leq m\leq +\infty$} consists of $\ringk{}$-linear morphisms $D_k\colon A\to B$, for $0\leq k\leq m$, such
that $D_0$ is a $\ringk{}$-algebra morphism and
\begin{equation}
  \label{eq:higherderivation}
  D_k(xy)=\sum_{i+j=k}D_i(x)D_j(y)\qquad x,y\in A.
\end{equation}
These equalities alone imply that $D_0$ preserves the product, so the requirement
that $D_0$ be an algebra morphism only means that $D_0(1)=1$.
It can be shown by induction that $D_k(1)=0$ for $k\geq 1$.

As a particular instance, a higher derivation of length $1$ is just an ordinary
derivation $D_1\colon A\to B$ as recalled in \cref{unimeascom}, where $B$ is regarded as an $A$-bimodule via
$D_0\colon A\to B$.

Following \cite{zbMATH05263295}, we denote the set of higher derivations of length
$m$ from $A$ to $B$ by $\Der^m_\ringk{}(A,B)$. There is a bijection between this set and
$\mathsf{Alg}_\ringk{}(A,B[x]/(x^{m+1}))$ when $m<\infty$, or
$\mathsf{Alg}_\ringk{}(A,B\llbracket x\rrbracket )$ when $m=+\infty$, sending
$(D_i:1\leq i\leq m)$ to the morphism $a\mapsto \sum_{i=0}^mD_i(a)x^i$.

The assignment $B\mapsto\Der^m_\ringk{}(A,B)$ is a functor
$\mathsf{Alg}_\ringk{}\to\mathsf{Set}$ via postcomposition. There is a $\ringk{}$-algebra $\HS^m_{A/\ringk{}}$ that represents this
functor, called the \emph{Hasse-Schmidt algebra} of $A$. The construction for a
commutative $\ringk{}$-algebra $A$ can be found in \cite{zbMATH05263295}.
Below we exhibit how it is related to the theory of measurings.

\subsection{Higher derivations and measurings}
\label{sec:high-deriv-meas}

In this section, $\ringk{}$ will remain a commutative ring.
Let $C_m$ be the $\ringk{}$-coalgebra with basis $\{v_i:0\leq i\leq m\}$ for $1\leq
m\leq+\infty$, with comultiplication $\Delta (v_k)=\sum_{i+j=k}v_i\otimes v_j$
and counit $\varepsilon(v_k)=\delta_{0,k}$.
Notice that the comultiplication is cocommutative.

\begin{lem}\label{lem:thislemma}
  There are a natural bijections
  \begin{equation*}
    \Der_{\ringk{}}^m(A,B)\cong \Meas(A,C_m,B)\cong \Coalg_\ringk{}(C_m,P(A,B))
  \end{equation*}
  natural in the $\ringk{}$-algebras $A$ and $B$.
\end{lem}
\begin{proof}
  The second natural isomorphism arises from the definition of $P(A,B)$, so we
  only have to exhibit the first one.
  A measuring $\mu\colon A\otimes C_m\to B$ can be described as morphisms
  $D_k\colon A\to B$, where $D_k(a)=\mu(a\otimes v_k)$. The left-hand side axiom~\eqref{measuring_mon} of a
  measuring for $\mu$ translates to the equality~\eqref{eq:higherderivation}. The
  right-hand side one, expressing the compatibility with the units and the counit, translates to
  $D_0(1)=1$ and $D_k(1)=0$ for $k>0$. Therefore, the $(D_k)$ form a higher
  derivation, and in fact this construction is a bijection, since the equalities
  $D_k(1)=0$ for $k\geq 1$ follows from~\eqref{eq:higherderivation}.
\end{proof}

It can be directly verified that the convolution algebra $[C_m, B]$ is
isomorphic to $B[x]/(x^{m+1})$ when $m$ is finite and to $B\llbracket
x\rrbracket$ when $m = +\infty$. This establishes the classical isomorphism
between
$\Der^m_\ringk{}(A, B)$ and $\Alg_\ringk{}(A, B[x]/(x^{m+1}))$ for finite $m$, or $\Alg_\ringk{}(A,
B\llbracket x\rrbracket)$ for $m = +\infty$, as explained in the commutative
case, for example, in \cite[Lemma~1.7]{zbMATH05263295}.

\subsection{The category of derivations}
\label{sec:category-derivations}
In this section we show how our general theory of enrichment directly leads to a
category $\mathsf{Der}^m_\ringk{}$ whose objects are the $\ringk{}$-algebras and whose
hom-sets are the sets $\Der^m_\ringk{}(A,B)$ of $\ringk{}$-derivations of order $m$.
To see this, notice that $C_m$, being a cocommutative
coalgebra, is a comonoid in the monoidal category $\Coalg_\ringk{}$. Therefore, the
functor $\Coalg_\ringk{}(C_m,-)\colon\Coalg_\ringk{}\to\mathsf{Set}$ has an induced lax
monoidal structure. Change of base along this functor sends
$\Coalg_\ringk{}$-categories to $\mathsf{Set}$-categories, and in particular it sends
$\underline{\Alg}_\ringk{}$ to a category with the same objects and with homs
$\Coalg_\ringk{}(C_m,P(A,B))\cong\Der^m_\ringk{}(A,B)$ by \cref{lem:thislemma}.
The composition of derivations $D\colon A\to B$ and $E\colon B\to C$ in this category
has the explicit form
$$
(E\circ D)_k=\sum_{i+j=k}E_i\circ D_j
$$
while the identity morphism of $A$ is the higher derivation that is
$(1_A,0,0,\dots)$.
We emphasise that this description of $\mathsf{Der}^m_\ringk{}$, as well as the fact
that the category axioms are satisfied, arise for free from the description of
the coalgebra $C_m$ and the fact that $\underline{\Alg}_\ringk{}$ is a category
enriched in coalgebras (\cref{monoidenrichment}).
Explicitly, given $D$, and $E$, we have correponding measurings
$d\colon A\otimes C_m\to B$ and $e\colon A\otimes C_m\to B$, and their
correponding coalgebra maps $\bar d\colon C_m\to P(A,B)$ and $\bar e\colon
C_m\to P(B,C)$. The composition of higher derivations $E\circ D$ is the higher
derivation associated to the coalgebra map
\begin{equation}
  \label{eq:5}
  C_m\xrightarrow{\Delta}C_m\otimes C_m\xrightarrow{\bar d\otimes \bar e}
  P(A,B)\otimes P(B,C)\to P(A,C)
\end{equation}
where we used that $C_m$ is commutative, so $\Delta$ is a coalgebra map, and the
last arrow is the composition of $\underline{\Alg}_\ringk{}$. The measuring
corresponding to \eqref{eq:5} is
\begin{equation*}
  A\otimes C_m\xrightarrow{1\otimes \Delta}A\otimes C_m\otimes C_m
  \xrightarrow{d\otimes 1}
  B\otimes C_m\xrightarrow{e}C
\end{equation*}
that sends $a\otimes v_k$ to $\sum_{i+j=k} e(d(a\otimes v_i)\otimes
v_j)=\sum_{i+j=k}E_jD_i(a)$.

There is a canonical identity-on-objects functor $\Alg_\ringk{}\to\mathsf{Der}^m_\ringk{}$ that
regards each morphism of algebras $f$ as a higher derivation
$(f,0,0,\dots)$. Furthermore, this functor has a retraction sending $(D_k)$ to
$D_0$. Indeed, the counit $\varepsilon\colon C_m\to \ringk{}$ has a section in $\Coalg_\ringk{} $ sending
$1$ to the group-like element $v_0$ of the basis of $C_m$. This gives rise to a retraction of monoidal natural
transformations $\Coalg_\ringk{}(C_m,-)\rightleftarrows\Coalg_\ringk{}(\ringk{},-)$, which in turn
induces $\Alg_\ringk{}\rightleftarrows\mathsf{Der}^m_\ringk{}$.

\subsection{The non-commutative Hasse-Schmidt algebra}
\label{sec:non-comm-hesse}

In this section we introduce the non-commutative version of the Hasse-Schmidt algebra of a $\ringk{}$-algebra $A$
in terms of the Sweedler product (namely the tensor of the enrichment
of \cref{monoidenrichment})
with
$C_m$. More precisely,
\begin{equation*}
  \HS_{A/\ringk{}}^m\cong C_m\triangleright A,\qquad 0\leq m\leq +\infty.
\end{equation*}
By definition of the Sweedler product, we have isomorphisms
\begin{equation}
  \label{eq:4}
  \Alg_\ringk{}(C_m\triangleright A,B)\cong \Coalg_\ringk{}(C_m,P(A,B))\cong\Der^m_\ringk{}(A,B)
\end{equation}
natural in $B$.

By the usual representability argument, there is an universal $m$-derivation
$(d^A_i\colon A\to \HS_{A/\ringk{}}^m;0\leq i\leq m)$ inducing
$\Alg_\ringk{}(\HS_{A/\ringk{}}^m,B)\cong\Der^m_\ringk{}(A,B)$ by composition, in the sense that for
any $m$-derivation $(D_i\colon A\to B:0\leq i\leq m)$ there exists a unique
algebra map $\HS_{A/\ringk{}}^m\to B$ such that $D_i=h\circ d^A_i$ for all $i$.
In particular, there is a commutative triangle as shown below, which we will
use later on.
\begin{equation}
  \label{eq:3}
  \begin{tikzcd}
    \Alg_\ringk{}(\HS_{A/\ringk{}}^m,B)\ar[dr,"(-\circ d_0^A)"']\ar[rr,"\cong"]&&
    \Der^m_\ringk{}(A,B)\ar[dl,"\mathrm{proj}"]\\
    &\Alg_\ringk{}(A,B)
  \end{tikzcd}
\end{equation}

The algebra map $d^A_0\colon A\to \HS^m_{A/\ringk{}}$ endows the Hasse-Schmidt algebra
of $\ringk{}\to A$ with an $A$-algebra structure.

\begin{lem}
  \label{l:HS-and-tensor}
  For any coalgebra $C$ and algebra $A$ over the commutative ring $\ringk{}$, there is
  a natural isomorphism
  $C\triangleright\HS^m_{A/\ringk{}}\cong \HS^m_{C\triangleright A/\ringk{}}$, for $0\leq
  m\leq+\infty$.
\end{lem}
\begin{proof}
  In general, $C\triangleright(D\triangleright A)$ is isomorphic to
  $(C\otimes D)\triangleright A$, and since the monoidal category $\Coalg_\ringk{}$ is
  symmetric, isomorphic to $D\triangleright (C\triangleright A)$. Now set
  $D=C_m$.
\end{proof}

\begin{lem}
  For $\ringk{}$-algebras $A$ and $B$, the set $\Der^m_\ringk{}(A,B)$ of derivations of order
  $m$ is naturally bijective with the set of group-like elements of a coalgebra,
  namely $P(\HS_{A/\ringk{}}^m,B)$, for $0\leq m\leq +\infty$.
\end{lem}
\begin{proof}
  The proof consitsts of the following string of isomorphisms:
  \begin{equation*}
    \Der^m_\ringk{}(A,B)\cong \Alg_\ringk{}(\HS_{A/\ringk{}}^m,B)\cong \Coalg_\ringk{}(\ringk{},P(\HS_{A/\ringk{}}^m,B)).
    \qedhere
  \end{equation*}
\end{proof}


We can apply the general theory established here to infer the behavior of
Hasse-Schmidt algebras under extension of scalars. In the commutative setting, this
is crucial for defining the scheme of jet differentials \cite{zbMATH05263295}.

\begin{lem}
  If $\ringk{}\to K$ is a morphism of commutative rings, then $\HS^m_{A/\ringk{}}\otimes_\ringk{}K
  \cong \HS^m_{A\otimes_\ringk{}K/K}$.
\end{lem}
\begin{proof}
  The restriction-of-scalars
  functor $W\colon \Mod_K\to\Mod_\ringk{}$ has a left adjoint $L=(-\otimes_\ringk{}K)$ that is
  strong monoidal and symmetric, giving rise to a monoidal adjunction
  $\Alg_K\rightleftarrows\Alg_\ringk{}$ that we still denote $L\dashv W$. If $C$ is a
  $\ringk{}$-module, we have
  $W[L(C),B]_K\cong [C,W(B)]_\ringk{}$ for all $B\in\Mod_K$, where the subindex
  indicates the hom of $K$ or $\ringk{}$-linear maps. Each $\ringk{}$-coalgebra structure on the
  $\ringk{}$-module $C$ induces a $K$-coalgebra structure on $L(C)$, since $L$ is
  strong monoidal. In the case a $K$-algebra $B$,
  the isomorphism $W[L(C),B]_K\cong [C,W(B)]_\ringk{}$ respects the convolution algebra
  structure on each side. This means that the diagram of right adjoints below
  commutes up to isomorphism, and, therefore, the diagram of left adjoints
  commute up to isomorphism too.
  \begin{equation*}
    \begin{tikzcd}[sep = large]
      \Alg_K\ar[d,"W"',shift right=5pt] \ar[r,"L(C)\triangleright-",shift
      left=5pt]&\Alg_K\ar[d,"W"',shift right= 5pt] \ar[l,"{[L(C),-]}",shift
      left=5pt]
      \\
      \Alg_\ringk{}\ar[r,"C\triangleright-",shift left=5pt] \ar[u,"L"',shift right=5pt]
      & \ar[l,"{[C,-]}",shift left=5pt] \Alg_\ringk{} \ar[u,"L"',shift right=5pt]
    \end{tikzcd}
  \end{equation*}
  We deduce that the natural transformation with components
  $L(C\triangleright A)\to L(C)\triangleright L(A)$ is an isomorphism.  The
  coalgebra $L(C_m)$ is directly shown to be isomorphic to the coalgebra $C_m$
  only
  now defined over $K$.  Then, we have obtained the natural isomorphism
  \begin{equation*}
    \HS^m_{A/\ringk{}}\otimes_\ringk{}K \cong \HS^m_{A\otimes_\ringk{}K/K}.
    \qedhere
  \end{equation*}
\end{proof}

The manuscript \cite{batchelor2021generalisedhomomorphismsmeasuringcoalgebras}
showed that, for a finite dimensional algebra $B$, then $B^\circ\cong B^*$ and
there exists a canonical isomorphism
$(B^\circ\triangleright A)^\circ\cong P(A,B)$.
\begin{cor}
  If $\ringk{}$ is a field and the algebra $B$ has finite dimension, then
  \begin{equation*}
    \Der^m_\ringk{}(A,B)\cong \Der^m_\ringk{}(B^\circ\triangleright A,\ringk{}).
  \end{equation*}
\end{cor}
\begin{proof}
  We actually have a stronger result. Indeed, $P(\HS_{A/\ringk{}}^{m},B)\cong
  P(B^\circ\triangleright \HS_{A/\ringk{}}^m,\ringk{})$ by
  \cite{batchelor2021generalisedhomomorphismsmeasuringcoalgebras}, which is
  isomorphic to $P(\HS_{B^\circ\triangleright A/\ringk{}}^m,\ringk{})$ by
  Lemma~\ref{l:HS-and-tensor}.
\end{proof}

\begin{rmk}
  In \cite{zbMATH06330791} the author investigates higher derivations $D\colon
  A\to A$ satisfying $D_0=\id$, and establish their connection to a well-known
  cocommutative Hopf algebra, which coincides with the free $\ringk{}$-algebra
  generated by our coalgebra $C_\infty$. Although this structure differs from
  the Hasse-Schmidt algebra introduced herein, it merits mentioning as an
  endeavor to devise an object that classifies a nontrivial class
  of higher derivations.
\end{rmk}

\subsection{The bimodule of K\"ahler differentials}
\label{sec:module-kahl-diff}

For a fixed $\ringk{}$-algebra $A$ (not necessarily commutative), we write $\Alg_A$ for the category $A\downarrow\Alg_\ringk{}$ of $A$-algebras and $\mathsf{Bimod}_A$ for the category of
$A$-bimodules. Each $A$-algebra can be regarded as an $A$-bimodule in an obvious
way. In other words, there is a functor $Z\colon \Alg_A\to\mathsf{Bimod}_A$
sending $(B,f\colon A\to B)$ to the $A$-bimodule $B$ obtained by restricting
scalars along $f$.

The category $\mathsf{Bimod}_A$ of bimodules over a $\ringk{}$-algebra $A$ has a
monoidal structure given by $\otimes_A$. The free monoid on a bimodule $M$ can
be constructed as the tensor algebra
$T_A(M)=\bigoplus_{n\geq 0}M^{\otimes_A ^n}$ which is an object of $\Alg_A$ via
the inclusion $A\to T_A(M)$ of the component of degree $0$. This gives a functor $T_A\colon\mathsf{Bimod}_A\to\Alg_A$, and
furthermore, $T_A\dashv Z$.

Consider the functor
$D_A\colon\mathsf{Bimod}_A\to\mathsf{Set}$ that sends each $A$-bimodule $M$ to the set
of ordinary derivations $A\to M$ (see the material preceding Example~\ref{ex:3} for the
precise definition). In what follows, we give a category-theoretical proof that
$D_A$ is representable---an alternative to the classical explicit
construction \cite[Prop.~III.10.17]{zbMATH03440338}.

\begin{lem}
  \label{l:Omega-exists}
  The functor $D_A$ is represented by a bimodule $\Omega_{A/\ringk{}}$.
\end{lem}
\begin{proof}
  The fact that filtered colimits in $\mathsf{Bimod}_A$ and $\Mod_\ringk{}$ are
  constructed as in $\mathsf{Set}$ is behind much of the proof that follows.
Since $D_A$ lands on $\mathsf{Set}$, for it to be representable it suffices that it has a left adjoint. Since its domain 
$\mathsf{Bimod}_A$ is locally finitely presentable, it is enough to show that $D_A$ is accessible and continuous; see the comments leading to
  \cite[Prop.~6.1.2]{MakkaiPare}. We can construct the set of derivations $D_A(M)$ as the equaliser in
  $\mathsf{Set}$ of
  \begin{equation}
    \label{eq:2}
    \begin{tikzcd}
      \Mod_\ringk{}(A,M)\ar[r,shift left,"\phi_M"]
      \ar[r,shift right,"\psi_M"']&
      M^{A\times A}
    \end{tikzcd}
    \qquad M\in\mathsf{Bimod}_A
  \end{equation}
  where $\phi_M(h)$ is the function $(x,y)\mapsto x\cdot h(y)+h(x)\cdot y$ and
  $\psi_M(h)$ is $(x,y)\mapsto h(xy)$. Both domain and codomain in~\eqref{eq:2}
  are accessible functors $\mathsf{Bimod}_A\to\mathsf{Set}$: they preserve
  $\kappa$-filtered colimits for a regular cardinal $\kappa$ larger than the
  cardinality of $A$---in fact, the functor $(\text{-})^{A\times A}$ preserves
  $\kappa$-filtered colimits for any $\kappa$ larger than the cardinality of
  $A\times A$, but for regular cardinals, which are infinite, this statement is
  equivalent to $\kappa$ being larger than the cardinality of $A$. Then, the
equaliser of~\eqref{eq:2} is accessible~\cite[Prop.~2.4.5]{MakkaiPare}. Preservation of
limits follows from a similar argument, as both domain and codomain
in~\eqref{eq:2} are continuous functors $\mathsf{Bimod}_A\to\mathsf{Set}$.
\end{proof}

The $A$-bimodule $\Omega_{A/\ringk{}}$, sometimes called the universal
first order differential calculus on $A$ \cite{zbMATH00003680}, is a generalisation of the module of
K\"ahler differentials of a commutative $\ringk{}$-algebra $A/\ringk{}$ (see for example
\cite{zbMATH00043569}).
One constructs $\Omega_{A/\ringk}$ for a non-commutative $\ringk$-algebra $A$ as
the kernel of the multiplication map $m \colon A \otimes_{\ringk} A \to A$, with
the universal derivation given by $a \mapsto a \otimes 1 - 1 \otimes a$
(see \cite[Prop.~III.10.17]{zbMATH03440338} and \cite[\S 2.4]{Buachalla}). As
observed in \cite[\S 2]{zbMATH01588500}, the commutativity of $\ringk$ required
in these references is not strictly necessary for the validity of the
construction.

We now show that the non-commutative Hasse-Schmidt algebra of an non-commutative
algebra $A$ is the free $A$-algebra over the $A$-bimodule $\Omega_{A/\ringk{}}$.
\begin{lem}
  \label{l:HSOmega}
  There is a canonical isomorphism $T_A(\Omega_{A/\ringk{}})\cong
  \operatorname{HS}_{A/\ringk{}}^1$.
\end{lem}
\begin{proof}
  Consider the functor $Z\colon \Alg_A\to \mathsf{Bimod}_A$ sending an
  $A$-algebra $(B,f\colon A\to B)$ to the $A$-bimodule $B$ described above. Then
  \begin{equation*}
    D_A Z(B,f)=\{\delta\colon A\to B: (f,\delta)\in\Der(A,B)\}
  \end{equation*}
  where $\Der(A,B)=\Der^1_{\ringk{}}(A,B)$ as explained in \cref{sec:higherdevs}.
  On a morphism $g\colon (B,f)\to (R,gf)$, we have
  $D_A(g)(f,\delta)=(gf,g\delta)$. By Lemma~\ref{l:Omega-exists} and the fact that $T_A\dashv Z$, there are
  natural isomorphisms
  \begin{equation*}
    D_A(Z(B,f))\cong\mathsf{Bimod}(\Omega_{A/\ringk{}},Z(B,f))\cong
    \Alg_A(T_A(\Omega_{A/\ringk{}}),(B,f)).
  \end{equation*}
  
  On the other hand,  $D_AZ\cong\Alg_A((\HS_{A/\ringk{}}^1,d_0),-)$. To see this,
  notice that
  $d_1\in D_A(\HS_{A/\ringk{}}^1,d_0)$ (since $(d_0,d_1)$ is a 1-derivation) and for
  any $\delta\in D_AZ(B,f)$ there exists a unique $g\colon \HS_{A/\ringk{}}^1\to B$ in
  $\Alg$ such that $gd_0=f$ and $gd_1=\delta$. So there exists a unique $g\colon
  (\HS_{A/\ringk{}}^1,d_0)\to (B,f)$ in $\Alg_A$ with $D_AZ(g)(d_1)=\delta$.

  Composing the natural isomorphisms we deduce
  $T_A(\Omega_{A/\ringk{}})\cong\HS_{A/\ringk{}}^1$ by Yoneda.
\end{proof}

\begin{rmk}
  The preceeding result shows that the non-commutative Hasse-Schmidt algebra
  $\operatorname{HS}_{A/\ringk{}}^1$ is a graded algebra isomorphic to the
  algebra of non-commutative differential forms $\Omega^\bullet(A)$ (\emph{c.f.}
  \cite[Prop.~2.3]{zbMATH00788747}). This is a differential-graded algebra
  central to non-commutative differential geometry. It has its origins in
  \cite{zbMATH03951325,zbMATH03880914,zbMATH03798472} and further developed in
  \cite{zbMATH00788747}; see also \cite{zbMATH07915441}.
\end{rmk}

The commutative Hasse-Schmidt algebra, which we denote by
$\operatorname{cHS}_{A/\ringk{}}^m$ to distinguish it from the non-commutative version,
is of central importance in the definition of jet spaces in Algebraic
Geometry~\cite{zbMATH05263295}. It is defined by the existence of natural
isomorphisms
\begin{equation}
  \label{eq:1}
  \mathsf{CAlg_\ringk{}}(\operatorname{cHS}_{A/\ringk{}}^m,B)\cong
  \Der^m_\ringk{}(A,B)
\end{equation}
where $A$ and $B$ are restricted to the subcategory
$\mathsf{CAlg}_\ringk{}\subset\Alg_\ringk{}$ of commutative algebras. The free commutative $A$-algebra on
an $A$-(bi)module $M$ is the symmetric algebra $S_A(M)=\bigoplus_{n\geq
  0}S_A^n(M)$ where $S_A^n(A)$ is the quotient of $M^{\otimes_A^n}$ by the action of
the symmetric group on $n$ elements.
\begin{cor}
  The commutative Hasse-Schmidt algebra $\operatorname{cHS}_{A/\ringk{}}^m$ is the
  abelianisation of $\operatorname{HS}_{A/\ringk{}}^m$.  Furtheremore,
  $\operatorname{cHS}_{A/\ringk{}}^1$ is isomorphic to the symmetric algebra
  $S_A(\Omega_{A/\ringk{}})$.
\end{cor}
\begin{proof}
  The first part of the statement follows from \eqref{eq:4} and \eqref{eq:1},
  while the second is Lemma~\ref{l:HSOmega} together with the fact that $S_A(M)$
  is the abelianisation of $T_A(M)$.
\end{proof}

The second part of the corollary recovers \cite[(1.4)]{zbMATH05263295} without
using any explict presentation of $\operatorname{cHS}^1_{A/\ringk{}}$.

\subsection{Higher derivations of modules}

In this section we interpret higher derivations of modules in terms of the theory
set forward in previous sections. Higher derivations of modules were introduced
in \cite{zbMATH03843957}, where universal modules classifying module derivations
are constructed in its \S 4. We show that these universal modules, sometimes
called Hasse-Schmidt modules, can be naturally constructed as tensor products in
the enriched category of modules presented in
Section~\ref{sec:enrichmentofmodincomod}.

Suppose that $(D_i\colon A\to B:i\geq 0)$ is a higher derivation of order
$0\leq m\leq +\infty$ of the $\ringk{}$-algebra $A$ in the $\ringk{}$-algebra $B$. If $M$ is a
right $A$-module and $N$ is a right $B$-module, a \emph{module derivation from
  $M$ to $N$ over $D$} of length $m$, or \emph{module $D$-derivation}, is a
sequence of $\ringk{}$-linear maps $(d_i\colon M\to N:0\leq i\leq m)$ such that
\begin{equation*}
  d_k(x\cdot a)=\sum_{i+j=k}d_i(x)\cdot D_j(a)\qquad x\in M,\ a\in A.
\end{equation*}
Then, a module derivation of length $0$ consists of an algebra derivation
$D_0\colon A\to B$ of
length $0$, which is to say an algebra map, together with a $\ringk{}$-linear map
$d_0\colon M\to N$ such that $d_0(x\cdot a)=d_0(x)\cdot D_0(a)$, for $x\in M$,
$a\in A$; in other words, $d_0$ is a module map over $D_0$.

A module derivation of length $m=1$ consists of an algebra map $D_0\colon A\to
B$, a derivation $D_1\colon A\to B$, where $B$ is regarded as an $A$-bimodule
via $D_0$, a module map $d_0$ over $D_0$ and a $\ringk{}$-linear map $d_1\colon M\to N$
such that $d_1(x\cdot a)=d_1(x)\cdot D_0(a)+ d_0(x)\cdot D_1(a)$; see Example~\ref{ex:5}.
These module derivations appear in the literature most often in the case when
$D_0=1_A$ and $d_0=1_M$, so we only have a derivation $\delta\colon A\to A$ and
a morphism $d\colon M\to M$ such that $d(x\cdot a)=d(x)\cdot a+ x\cdot \delta(a)$.

We denote the set of module derivations from $M_A$ to $N_B$ of length $m$ by
$\MDer^m_\ringk{}(M_A,N_B)$.
Below we show that these are the hom-sets of a category. We proceed in the same
fashion we did in the previous section for higher (algebra) derivations.

The coalgebra $C_m$ that we used to relate measurings with higher derivations in
the previous section can be regarded as a right comodule over itself. 

\begin{prop}
  There exist bijections
  \begin{equation*}
    \MDer^m_\ringk{}(M_A,N_B)\cong\Meas(M_A,C_m,N_B)
    \cong \Comod(C_m,Q(M_A,N_B))\cong \Mod(C_m\triangleright M_A,N_B)
  \end{equation*}
  natural in $M_A,N_B\in\Mod$. There is a category with objects given by modules
  and
  homs given by the sets $\MDer^m_\ringk{}(M_A,N_B)$.
\end{prop}
\begin{proof}
  A $C_m$-module measuring consists of an algebra measuring $\phi\colon A\otimes
  C_m\to B$ with a $\ringk{}$-linear map $\psi\colon M\otimes C_m\to B$ that satisfies
  \begin{equation*}
    \psi(u\cdot x\otimes v_k)=\sum_{i+j=k}\psi(u\otimes v_i)\cdot\phi(x\otimes
    v_j)\qquad u\in M,x\in A
  \end{equation*}
  This is to say that the maps $d_k=\psi(-\otimes v_k)\colon M\to N$ form a
  module $D$-derivation where $D=(D_k=\phi(-\otimes v_k))$ is the higher
  derivation that corresponds to the measuring $\phi$.
  The rest of the isomorphisms are the definition of $Q(M_A,N_B)$ and of the
  tensor of a module by a comodule. Finally, the $C_m$-comodule $C_m$ is a
  comonoid in $\Comod$ by virtue of its cocommutativity. Therefore,
  $\Comod(C_m,-)$ is a lax monoidal functor $\Comod\to\mathsf{Set}$ which
  therefore induces an ordinary category from the $\Comod$-category of modules,
  with homs $\Comod(C_m,Q(M_A,N_B))$.
\end{proof}

The module $C_m\triangleright M_A$ is defined over $C_m\triangleright A\cong
\HS_{A/\ringk{}}^m$, it classifies module derivations and may be called the \emph{Hasse-Schmidt
module} of $M_A$.

\bibliographystyle{plain}
\bibliography{references}

\end{document}